\documentclass[preprint,12pt]{elsarticle}
\usepackage{lineno,hyperref}
\usepackage [latin1]{inputenc}
\usepackage{latexsym}
\usepackage[left=1in,right=1in,top=1in,bottom=1in]{geometry}
\usepackage{amsmath}
\numberwithin{equation}{section}
\usepackage{amssymb}
\usepackage{ntheorem}
\journal{}
\newtheorem{theorem}{Theorem}[section]
\newtheorem{lemma}{Lemma}[section]
\newtheorem*{proof}{Proof}
\newtheorem{example}{Example}[section]
\newtheorem{remark}{Remark}[section]
\numberwithin{figure}{section}
\numberwithin{table}{section}

\usepackage{charter}
\usepackage[charter]{mathdesign}
\usepackage{algorithm}
\usepackage{algorithmicx}
\usepackage{algpseudocode}
\usepackage{booktabs}
\usepackage{graphicx}
\usepackage{subfigure}
\usepackage{epstopdf}
\usepackage{multirow}
\usepackage{algorithm,algpseudocode,float}
\usepackage{lipsum}
\usepackage{makecell}
\numberwithin{figure}{section}
\numberwithin{algorithm}{section}

\begin{document}
\begin{frontmatter}
\markboth{Li Weiguo, Bao Wendi and Xing Lili}{Kaczmarz-Type Method for Solving Matrix Equation $AXB=C$}
\title{Kaczmarz-Type Method for Solving Matrix Equation $AXB=C$}
\author[mymainaddress]{Xing Lili}
\author[mymainaddress]{Bao Wendi}
\author[mymainaddress]{Li Weiguo}
\cortext[mycorrespondingauthor]{Corresponding author}
\ead{liwg@upc.edu.cn}
\address[mymainaddress]{College of Science, China University of Petroleum, Qingdao 266580, P .R. China}
\begin{abstract}
In this paper, several row and column orthogonal projection methods are proposed for solving matrix equation $AXB=C$, where the matrix $A$ and $B$ are full rank or rank deficient and equation is consistent or not. These methods are iterative methods without matrix multiplication. It is theoretically proved these methods converge to the solution or least-squares solution of the matrix equation. Numerical results show that these methods are more efficient than iterative methods involving matrix multiplication for high-dimensional matrix.
\end{abstract}
\begin{keyword}
 matrix equation\sep Kaczmarz method\sep Gauss-Seidel method\sep convergence
\end{keyword}
\end{frontmatter}
\section{Introduction}
Consider the linear matrix equation
 \begin{equation}\label{e11}
AXB=C,
\end{equation}
where $A\in R^{m\times p}$, $B\in R^{q\times n}$ and $C\in R^{m\times n}$. Such problems arise in linear control and filtering theory for continuous or discrete-time large-scale dynamical systems. They play an important role in image restoration and other problems; for more details see \cite{BJ08, ZD08} and the references therein. The linear matrix matrix Eq. (\ref{e11}) has been considered by many authors. In \cite{P56} Penrose presented a sufficient and necessary condition for the consistency of this equation and for the consistent case, he provided a representation of its general solution. When the matrices $A$ and $B$ are small and dense, direct methods such as $QR$-factorization-based algorithms \cite{FF94, Z95} are attractive. However, these direct algorithms are quite costly and impractical when $A$ and $B$ are large. Therefore, iteration methods in \cite{DC05, DLD08, WLD13, TTLX17} to solve the matrix Eq. (\ref{e11}) have attracted much interests recently. Many methods among these frequently use the matrix-matrix product operation, and consume a lot of computing time.

It is well known that the matrix Eq. (\ref{e11}) can be written the following mathematically equivalent matrix-vector form by Kronecker products symbol
\begin{equation}\label{e12}
(B^T\otimes A){\rm vec}(X)={\rm vec}(C).
\end{equation}
where the Kronecker product $B^T\otimes A\in R^{mn\times pq}$, the right-side vector vec$(C)\in R^{mn\times 1}$, and unknown vector vec$(X)\in R^{pq\times 1}$. With the application of Kronecker products, many algorithms are proposed to solve the matrix Eq. (\ref{e11}) (see, e.g., \cite{ZLGZ11, CI08, P10}). However, when the dimensions of matrices $A$
and $B$ are large, the dimensions of the linear system (\ref{e12}) increase dramatically, which increases memory usage and computational cost of the numerical algorithms to find an approximate solution of Equation (\ref{e12}). Du et al. proposed the randomized block coordinate descent (RBCD) method for solving the matrix least-squares problem $\max\limits_{X\in R^{p\times q}}\|C-AXB\|_F$ in \cite{DRS22}. This method requires that matrix B is full row rank. Wu et al. \cite{WLZ22} introduced two kinds of Kaczmarz-type methods to solve consistent matrix equation $AXB=C$: relaxed greedy randomized Kaczmarz (ME-RGRK) and maximal weighted residual Kaczmarz (ME-MWRK). Although the row and column index selection strategy is time-consuming, the ideas of these two methods are suitable for solving large-scale consistent matrix equations.
In \cite{NZ22}, Niu and Zheng proposed two classes of global randomized Kaczmarz methods: the global randomized block Kaczmarz algorithm and global randomized average block
Kaczmarz algorithm, for solving large-scale consistent linear matrix equation $AXB=C$.

In this paper, Kaczmarz method \cite{K37} and coordinate descent method \cite{W15} are used to solve (\ref{e11}) (maybe consistent or inconsistent) by the product of matrix and vector.

All the results in this paper hold in the complex field. But for the sake of simplicity, we only discuss it in the real number field.

In this paper, we denote $A^T$,  $A^+$, $r(A)$, $R(A)$, $\|A\|_F= \sqrt{{\rm trace}(A^TA)}$ and $\langle A, B\rangle_F= {\rm trace}(A^TB)$ as the transpose, the Moore-Penrose generalized inverse, the rank of $A$, the column space of $A$, the Frobenius norm of $A$ and the inner product of two matrices $A$ and $B$, respectively. We use $I$ to denote the identity matrix whose order is clear from the context. In addition, for a given matrix $G=(g_{ij})\in R^{m\times n}$, $G_{i,:}$, $G_{:,j}$, $\sigma_{\max}(G)$ and $\sigma_{\min}(G)$, are used to denote its $i$th row, $j$th column, the maximum singular value  and the smallest nonzero singular value of $G$ respectively. Let $E_k$ denote the expected value conditional on the first $k$ iterations, that is,
$$E_k[\cdot]=E[\cdot|i_0,j_0,i_1,j_1,...,i_{k-1},j_{k-1}],$$
where $i_s$ and $j_s(s=0,1,...,k-1)$ are the row and the column chosen at the $s$th iteration. Let the conditional expectations with respect to the random row index be  $$E_k^i[\cdot]=E[\cdot|i_0,j_0,i_1,j_1,...,i_{k-1},j_{k-1},j_k] $$
and with respect to the random column index be  $$E_k^j[\cdot]=E[\cdot|i_0,j_0,i_1,j_1,...,i_{k-1},j_{k-1},i_k]. $$ By the law of total expectation, it holds that $E_k[\cdot] =E_k^i [E_k^j[\cdot]] $.

The organization of this paper is as follows. In Section 2, we will discuss Kaczmarz method for finding the minimal $F$-norm solution of the consistent matrix Eq. (\ref{e11}). In Section 3, we give a Gauss-seidel method for solving the least-squares solution of the inconsistent matrix Eq.  (\ref{e11}). In Section 4, we discuss the extended Kaczmarz method and extended Gauss-Seidel method for finding the minimal $F$-norm least-squares solution of the matrix Eq. (\ref{e11}). In Section 5, some numerical examples are provided to illustrate the effectiveness of our new methods. Finally, some brief concluding remarks are described in Section 6. We summarize the convergence of the proposed methods in expectation to the minimal $F$-norm solution $X^*=A^+CB^+$ for all types of matrix equations in Table \ref{tab11}.
\begin{table}[H]
\caption{Summary of the convergence of ME-GRGK \cite{WLZ22}, ME-WMRK \cite{WLZ22}, RBCD \cite{DRS22}, CME-RK (Theorem \ref{t21}), IME-RGS (Theorem \ref{t31}), IME-REKRK (Theorem \ref{t41}), IME-REKRGS (Theorem \ref{t42}), DREK  and DREGS in expectation to the minimal $F$-norm solution $X^*=A^+CB^+$ for all types of matrix equations. (Note. Y means the algorithm is convergent and N means not.)}
\label{tab11}
\centering
\resizebox{\textwidth}{!}{
\begin{tabular}{ c c c c c c c c c c c c c}
\hline
  matrix equation  & r(A)  &r(B) & ME-GRGK  & ME-WMRK  & RBCD  & CME-RK & IME-RGS  &  IME-REKRK & IME-REKRGS & DREK   & DREGS  \\
\hline
\multirow{4}{*}{consistent}  & =p & =q & Y & Y & Y &Y & Y &Y   & Y   & Y	&  Y	\\
    & =p   & <q   &Y  & Y  & N  & Y & N & Y &  Y (r(B)=n)  & Y	&  Y		\\
     &<P   &=q    &  Y  & Y  & Y  & Y & N & N &  Y  & Y	&  Y		\\
      & <p   & <q   & Y   & Y  & N &  Y& N &  Y(r(B)=n) &   N   & Y	&  Y		\\
\hline
\multirow{4}{*}{inconsistent}  & =p & =q & N  & N  & Y  &N  & Y &Y   & Y   & Y	&  Y	\\
   & =p   & <q    & N   & N  & N   & N & N &   Y(r(B)=n) &  N    & Y	&  Y 		\\
     &<P   &=q  & N   & N  & Y  & N & N & N &  Y   & Y	&  Y		\\
      & <p   & <q  & N  & N  & N & N & N & Y(r(B)=n) &  N   & Y	&  Y	\\
\hline
 \end{tabular}}
\end{table}

\section{Kaczmarz Method for Consistent Case}
If the matrix Eq. (\ref{e11}) is consistent, i.e., $AA^+CB^+B=C$ (necessary and sufficient conditions for consistent, hence $A^+CB^+$ is a solution of the consistent matrix Eq.  (\ref{e11})). Especially, if $A$ is full row rank ($m\leq p$) and $B$ is full column rank ($q\geq n$), the matrix Eq.  (\ref{e11}) is consistent because $X^*=A^T(AA^T)^{-1}C(B^TB)^{-1}B^T$ is one solution of this equation.
In general, the matrix Eq. (\ref{e11}) has multiple solutions. Now we try to find its minimal $F$-norm solution $X^*=A^+CB^+$ by Kaczmarz method.

Assume that $A$ has no row of all zeros and $B$ has no column of all zeros. The matrix Eq.  (\ref{e11}) can be rewritten as the following system of matrix equations
\begin{equation}\label{e21}
\left\{
\begin{array}{c}
AY=C,\\
B^TX^T=Y^T,
\end{array}
\right.
\end{equation}
where $Y\in R^{p\times n}$.
The classical Kaczmarz method which was introduced in 1937 \cite{K37} is a row projection iterative algorithm
for solving a consistent system $Ax = b$ where $A\in R^{m\times p}$, $b\in R^m$ and $x\in R^p$. This method involves only a single equation per iteration as follows which converges to the least norm solution $A^+b$ of $Ax = b$ with a initial iteration $x^{(0)}\in R(A^T)$,
\begin{equation}\label{e22}
x^{(k+1)}=x^{(k)}+\frac{b_i-A_{i,:}x^{(k)}}{\|A_{i,:}\|_2^2}A_{i,:}^T, \  \ k\ge 0,
\end{equation}
where $i=(k \ mod \ m) + 1$. If we iterate the system of linear equations $AY_{:,j}=C_{:,j}$, $j=1, \cdots, n$ simultaneously and denote $Y^{(k)}=[Y_{:,1}^{(k)}, Y_{:,2}^{(k)}, \cdots, Y_{:,n}^{(k)}]$, we get
\begin{equation}\label{e23}
Y^{(k+1)}=Y^{(k)} +\frac{A_{i,:}^T}{\|A_{i,:}\|_2^2}(C_{i,:}-A_{i,:}Y^{(k)}), \  \ k\ge 0,
\end{equation}
where $i=(k \ mod \ m) + 1$. And then $A_{i,:}Y^{(k+1)}=C_{i,:}$ holds, that is, $Y^{(k+1)}$ is the projection of $Y^{(k)}$ onto the subspace $H_i=\left\{Y\in R^{p\times n}: \ A_{i,:}Y=C_{i,:}\right\}$. So we obtain an orthogonal projection method to solve the matrix equation $AY=C$.

Similarly, we can get the following column orthogonal projection method to solve equation $B^TX^T=(Y^{(k+1)})^T$
\begin{equation}\label{e24}
X^{(k+1)}=X^{(k)}+\frac{Y_{:,j}^{(k+1)}-X^{(k)}B_{:,j}}{\|B_{:,j}\|^2_2}B_{:,j}^T, \ k\ge 0,
\end{equation}
where $j=(k\ mod\ n)+1$. And then $X^{(k+1)}B_{:,j}=Y_{:,j}^{(k+1)}$ holds, that is, $X^{(k+1)}$ is the projection of $X^{(k)}$ onto the subspace $\hat{H}_j=\left\{X\in R^{p\times q}: XB_{:,j}=Y_{:,j}^{(k+1)}\right\}$.

With the use of the formulae (\ref{e23}) and (\ref{e24}), we get a randomized Kaczmarz-type algorithm as follows, which is called the CME-RK algorithm.

\begin{algorithm}
  \leftline{\caption{RK Method for Consistent Matrix Equation $AXB=C$ (CME-RK)\label{alg21}}}
  \begin{algorithmic}[1]
    \Require
      $A\in R^{m\times p}$, $B\in R^{q\times n}$, $C\in R^{m\times n}$,$(X^{(0)}_{i,:})^T\in R(B),\ i=1,\ldots,p$, $Y^{(0)}=X^{(0)}B$, $Y^{(0)}_{:,j}\in R(A^T),\ j=1,\ldots, n$, $K\in R$
    \State For $i=1:m$, $M(i)=\|A_{i,:}\|_2^2$
    \State For $j=1:n$, $N(j)=\|B_{:,j}\|_2^2$
    \For {$k=0,1,2,\cdots, K-1$}
    \State Pick $i$ with probability $p_{i}(A)=\frac{\|A_{i,:}\|_2^2}{\|A\|^2_F}$   and $j$ with probability $\hat{p}_{j}(B)=\frac{\|B_{:,j}\|_2^2}{\|B\|^2_F}$
    \State Compute $Y^{(k+1)}=Y^{(k)}+\frac{A_{i,:}^T}{M(i)}(C_{i,:}-A_{i,:}Y^{(k)})$
    \State Compute $X^{(k+1)}=X^{(k)}+\frac{Y_{:,j}^{(k+1)}-X^{(k)}B_{:,j}}{N(j)}B_{:,j}^T$
    \EndFor
    \State Output $X^{(K)}$
  \end{algorithmic}
\end{algorithm}

The cost of each iteration of this method is $4p(n+q)+2p$ if the square of the row norm of $A$ and the square of the column norm of $B$  are pre-computed in advance. In the following theorem, with the idea of the RK method \cite{SV09}, we will prove that $X^{(k)}$ generated by Algorithm \ref{alg21}  converges to the the minimal $F$-norm solution of $AXB=C$ if $i$ and $j$ are picked at random.

Before proving the convergence result of Algorithm \ref{alg21}, we analyze the convergence of $Y^{(k)}$ and give the following lemmas. Let $Y^*= A^+C$. The sequence $\{Y^{(k)}\}$ is generated by  (\ref{e23}) starting from the initial matrix $Y^{(0)}\in R^{p\times n}$.
\begin{lemma}\label{lem1}
If the sequence $\{Y^{(k)}\}$ is convergent, it must converge to $Y^*=A^+C$, provided that $Y^{(0)}_{:,j}\in R(A^T),\ j=1,\ldots, n$.
\end{lemma}
\begin{proof}
Let $\tilde{Y}$ be the limit point of $\{Y^{(k)}\}$. From the consistency of the iteration scheme, we have
$$ A^T_{i,:}(A_{i,:}\tilde{Y}-C_{i,:})=0, \ i=1,2,\ldots, m. $$
As a result, it holds $A^T A\tilde{Y} - A^T C = \sum_{i=1}^m A^T_{i,:}(A_{i,:}\tilde{Y}-C_{i,:})=0$, that is $A^T A\tilde{Y} = A^T C$. This means that $\tilde{Y}_{:,j}$ is the least-square solution of $Ax=C_{:,j}, j=1,2,\ldots, n$. Thus $\tilde{Y}_{:,j}$ can be written as
\begin{equation}\label{e25}
\tilde{Y}_{:,j}=A^+C_{:,j}+(I-A^+A)z, z\in R^p, j=1,2,\ldots, n.
\end{equation}
Noting that $A^T_{i,:}=A^T I_{:,i}$, for $k\ge 0$, Kaczmarz method (\ref{e23}) can be rewritten as
\begin{equation*}
Y^{(k+1)}_{:,j}=Y^{(k)}_{:,j} +\frac{C_{i,j}-A_{i,:}Y^{(k)}_{:,j}}{\|A_{i,:}\|_2^2}A^T I_{:,i}, \  \ j=1,2,\ldots, n.
\end{equation*}
At each iteration, the above method only adds a linear combination of a column of $A^T$ to corresponding column of the current iterate, so  $Y^{(k)}_{:,j} \in R(A^T)$, provided  $Y^{(0)}_{:,j}\in R(A^T),\ j=1,\ldots, n$. Then $\tilde{Y}_{:,j} \in R(A^T), j=1,2,\ldots, n$. Combining (\ref{e25}), we get $\tilde{Y}_{:,j}=A^+C_{:,j}, j=1,2,\ldots, n$, so that $\tilde{Y}\equiv Y^*=A^+C$. This completes the proof.
\end{proof}
\begin{lemma}\label{lem3}
 Let $A\in R^{m\times p}$ be any nonzero matrix. For $Y_{:,j}\in R(A^T),\ j=1,\ldots, n$, it holds that
 \begin{equation*}
\|AY\|_F^2\geq\sigma_{\min}^2(A)\|Y\|_F^2.
\end{equation*}
\end{lemma}
\begin{lemma}\label{th1}
The sequence $\{Y^{(k)}\}$ generated by  (\ref{e23}) starting from the initial matrix $Y^{(0)}\in R^{p\times n}$ in which $Y^{(0)}_{:,j}\in R(A^T),\ j=1,\ldots, n$, converges linearly to $A^+C$ in mean square form. Moreover, the solution error in expectation for the iteration sequence $Y^{(k)}$ obeys
\begin{equation}\label{AY=C}
E\left[\left\|Y^{(k)}-A^+C\right\|_F^2\right]\leq \rho_1^{k}\left\|Y^{(0)}-A^+C\right\|_F^2,
\end{equation}
where the $i$th row of $A$ is selected with probability $p_i(A)=\frac{\|A_{i,:}\|_2^2}{\|A\|^2_F}$, and $\rho_1=1-\frac{\sigma^2_{\min}(A)}{\|A\|^2_F}$.
\end{lemma}
\begin{proof}
It is easy to see that
\begin{equation*}
\left\| Y^{(k)} -Y^*\right\|_F^2=\|Y^{(k)}-Y^{(k+1)}\|_F^2+\left\|Y^{(k+1)} - Y^*\right\|_F^2-2\langle Y^{(k+1)}-Y^{(k)}, Y^{(k+1)}-Y^*\rangle_F.
\end{equation*}
It follows from
\begin{align*}
\langle Y^{(k+1)}-Y^{(k)}, Y^{(k+1)}-Y^*\rangle_F=&\left\langle \frac{A_{i,:}^T}{\|A_{i,:}\|_2^2}(C_{i,:}-A_{i,:}Y^{(k)}), Y^{(k+1)}-Y^*\right\rangle_F \\
=&{\rm trace}\left(\frac{1}{\|A_{i,:}\|_2^2}(C_{i,:}-A_{i,:}Y^{(k)})^TA_{i,:}(Y^{(k+1)}-Y^*)\right)\\
=&0  \ ({\rm by} \ A_{i,:}Y^{(k+1)} = C_{i,:} \ {\rm and\ } A_{i,:}Y^*=C_{i,:} )
\end{align*}
and
\begin{align*}
\|Y^{(k)}-Y^{(k+1)}\|_F^2=& \left\|\frac{A_{i,:}^T}{\|A_{i,:}\|^2_2}\left(  A_{i,:} Y^{(k)} -C_{i,:} \right)\right\|_F^2\\
  & =  {\rm trace}\left(\left(  A_{i,:} Y^{(k)} -C_{i,:} \right)^T \frac{A_{i,:}}{\|A_{i,:}\|^2_2} \frac{A_{i,:}^T}{\|A_{i,:}\|^2_2}\left(  A_{i,:} Y^{(k)} -C_{i,:} \right)  \right)\\
  & = \frac{\left\|  A_{i,:} Y^{(k)} -C_{i,:} \right\|_2^2}{\|A_{i,:}\|^2_2} \ ({\rm by\ } {\rm trace} (uu^T)=\|u\|_2^2 {\rm \ for\  } \forall u\in R^{n})
\end{align*}
that
\begin{equation*}
\left\|Y^{(k+1)} - Y^*\right\|_F^2= \left\| Y^{(k)} -Y^*\right\|_F^2 - \frac{\left\|  A_{i,:} Y^{(k)} -C_{i,:} \right\|_2^2}{\|A_{i,:}\|^2_2}.
\end{equation*}
By taking the conditional expectation, we have
\begin{align}\label{eq15}
E_k\left[\left\|Y^{(k+1)} - Y^*\right\|_F^2\right]=&E_k\left[\left\| Y^{(k)} -Y^*\right\|_F^2-\frac{\left\|  A_{i,:} Y^{(k)} -C_{i,:} \right\|_2^2}{\|A_{i,:}\|^2_2}\right] \notag\\
  & =  \left\| Y^{(k)} -Y^*\right\|_F^2 - \sum\limits_{i=1}^m \frac{\|A_{i,:}\|^2_2}{\|A\|_F^2}\frac{\left\|  A_{i,:} Y^{(k)} -C_{i,:} \right\|_2^2}{\|A_{i,:}\|^2_2} \notag\\
  & = \left\| Y^{(k)} -Y^*\right\|_F^2- \frac{\left\|A Y^{(k)} -AY^*\right\|_F^2}{\|A\|_F^2} \notag\\
  & \le \left\| Y^{(k)} -Y^*\right\|_F^2- \frac{\sigma^2_{\min}(A)}{\|A\|_F^2}\left\| Y^{(k)} -Y^*\right\|_F^2 \ ( {\rm by\  Lemma}\ \ref{lem3}) \notag\\
 & =  \left ( 1- \frac{\sigma^2_{\min}(A)}{\|A\|_F^2}\right) \left\| Y^{(k)} - Y^*\right\|_F^2, \ k\ge 0.
\end{align}
Finally, by  (\ref{eq15}) and induction on the iteration index $k$, we straightforwardly obtain the estimate (\ref{AY=C}). This completes the proof.
\end{proof}

Similarly, we can get the following convergence result of RK method for the matrix equation $ B^TX^T=( Y^*)^T $.
\begin{lemma}\label{lem4}
Let $X^*= A^+CB^+$. $\tilde{X}$ is generated by running one-step RK update for solving the matrix equation $ B^TX^T=( Y^*)^T $ starting from any matrix $\hat{X}\in R^{p\times q}$ in which $(\hat{X}_{i,:})^T\in R(B),\ i=1,\ldots,p$. Then it holds
\begin{equation}\label{e28}
E\left[\left\|\tilde{X}-A^+CB^+\right\|_F^2\right]\leq \rho_2\left\|\hat{X}-A^+CB^+\right\|_F^2,
\end{equation}
where the $j$th column of $B$ is selected with probability $\hat{p}_j(B)=\frac{\|B_{:,j}\|_2^2}{ \|B \|^2_F}$ and $\rho_2=1-\frac{\sigma^2_{\min}(B)}{ \|B \|^2_F}$.
\end{lemma}
\begin{lemma}\label{lem2}
Let $ \tilde{H}_j=\left\{X\in R^{p\times q}: XB_{:,j}=Y^*_{:,j}\right\} $ be the subspaces consisting of the solutions to the unperturbed equations, and let $ \hat{H}_j=\left\{X\in R^{p\times q}: XB_{:,j}=Y^{(k+1)}_{:,j}\right\} $ be the solutions spaces of the noisy equations. Then $ \hat{H}_j=\left\{ W+ \alpha_j^{(k+1)} B^T_{:,j}, W \in \tilde{H}_j \right\} $, where $\alpha_j^{(k+1)} = \frac{Y^{(k+1)}_{:,j}-Y^{*}_{:,j}}{\|B_{:,j}\|^2_2}$.
\end{lemma}
\begin{proof}
First, if $W\in \tilde{H}_j$, then $$ (W+ \alpha_j^{(k+1)}  B^T_{:,j})B_{:,j} =W B_{:,j} + \alpha_j^{(k+1)} B^T_{:,j} B_{:,j} = Y^*_{:,j}+Y^{(k+1)}_{:,j} - Y^*_{:,j}= Y^{(k+1)}_{:,j},$$so $W+ \alpha_j^{(k+1)} B^T_{:,j} \in \hat{H}_j$.

Next, let $V \in \hat{H}_j$. Set $W= V-\alpha_j^{(k+1)} B^T_{:,j}$, then $$WB_{:,j}=(V-\alpha_j^{(k+1)} B^T_{:,j})B_{:,j}=VB_{:,j}-\alpha_j^{(k+1)} \|B_{:,j}\|_2^2 =Y^{(k+1)}_{:,j}-( Y^{(k+1)}_{:,j}-Y^{*}_{:,j})  = Y^*_{:,j},$$ so $W\in \tilde{H}_j$. This completes the proof.
\end{proof}

We present the convergence result of Algorithm \ref{alg21} in the following theorem.
\begin{theorem}\label{t21}
The sequence $\{X^{(k)}\}$ generated by Algorithm \ref{alg21} starting from the initial matrix $X^{(0)}\in R^{p\times q}$ and $Y^{(0)}=X^{(0)}B$, converges linearly to the solution $X^*=A^+CB^+$ of the consistent  matrix Eq. (\ref{e11}) in mean square if $(X^{(0)}_{i,:})^T\in R(B),\ i=1,\ldots,p$ and $Y^{(0)}_{:,j}\in R(A^T),\ j=1,\ldots, n$. Moreover, the following relationship holds
\begin{equation}
E \left[\left\|X^{(k)}-A^+CB^+\right\|_F^2 \right]
 \le \left(1+\frac{ \sigma^2_{\max}(B)\eta }{\|B \|^2_F}\right) \rho_2^{k} \left\|X^{(0)}-A^+CB^+\right\|_F^2,
\end{equation}
where the $i$th row of $A$ is selected with probability $p_i(A)=\frac{\|A_{i,:}\|_2^2}{\|A \|^2_F}$, the $j$th column of $B$ is selected with probability $\hat{p}_j(B)=\frac{\|B_{:,j}\|_2^2}{ \|B \|^2_F}$, and $
 \eta= \left\{
\begin{array}{ll}
\frac{ \rho_1 }{\rho_2-\rho_1}, & if \  \rho_1<\rho_2,\\
\frac{ \rho_1 }{\rho_1-\rho_2}\left[ \left(\frac{\rho_1}{\rho_2}\right)^{k}-1\right], & if \  \rho_1>\rho_2,\vspace{0.2cm}\\
{k}, & if \  \rho_1=\rho_2.
\end{array}
 \right.
$
\end{theorem}
\begin{proof}
Let $X^{(k)}$ denote the $k$th iterate of randomized Kaczmarz method (\ref{e24}), and $\hat{H}_j$ be the solution space chosen in the $(k+1)$th iteration. Then $X^{(k+1)}$ is the orthogonal projection of $X^{(k)}$ onto $\hat{H}_j$. Let $\tilde{X}^{(k+1)}$ denote the orthogonal projection of $X^{(k)}$ onto $\tilde{H}_j$. By using (\ref{e24}) and Lemma \ref{lem2}, we have that
\begin{align*}
X^{(k+1)} & = X^{(k)}+\frac{Y_{:,j}^{(k+1)}-X^{(k)}B_{:,j}}{\|B_{:,j}\|^2_2}B_{:,j}^T\\
& = X^{(k)}+\frac{Y^*_{:,j}-X^{(k)}B_{:,j}}{\|B_{:,j}\|^2_2}B_{:,j}^T + \frac{Y_{:,j}^{(k+1)}-Y^*_{:,j}}{\|B_{:,j}\|^2_2}B_{:,j}^T\\
& = \tilde{X}^{(k+1)} + \frac{Y_{:,j}^{(k+1)}-Y^*_{:,j}}{\|B_{:,j}\|^2_2}B_{:,j}^T.
\end{align*}
Then
\begin{align*}
\langle X^{(k+1)}- \tilde{X}^{(k+1)},  \tilde{X}^{(k+1)}-X^* \rangle_F & =\left\langle\frac{Y_{:,j}^{(k+1)}-Y^*_{:,j}}{\|B_{:,j}\|^2_2}B_{:,j}^T,  \tilde{X}^{(k+1)}-X^* \right\rangle_F \\
& = {\rm trace}\left( B_{:,j} \frac{( Y_{:,j}^{(k+1)}-Y^*_{:,j})^T}{\|B_{:,j}\|^2_2}(\tilde{X}^{(k+1)}-X^* ) \right) \\
& = {\rm trace}\left((\tilde{X}^{(k+1)}-X^* ) B_{:,j} \frac{( Y_{:,j}^{(k+1)}-Y^*_{:,j})^T}{\|B_{:,j}\|^2_2} \right)\\
 & \ \ \ \ ({\rm \ by \ }{\rm trace}(MN)={\rm trace}(NM)\ {\rm for \ any \ matrices\ }M,N )\\
 & =0 \ ({\rm \ by \ } \tilde{X}^{(k+1)} B_{:,j} = Y^*_{:,j}, X^* B_{:,j}= Y^*_{:,j} ),
\end{align*}
and
\begin{align*}
\left\|X^{(k+1)}-\tilde{X}^{(k+1)}\right\|_F^2  & =\left \|\frac{ Y_{:,j}^{(k+1)}-Y^*_{:,j}}{\|B_{:,j}\|^2_2} B_{:,j}^T\right\|_F^2 = \frac{1}{\|B_{:,j}\|_2^4}{\rm trace}( B_{:,j}( Y_{:,j}^{(k+1)}-Y^*_{:,j})^T ( Y_{:,j}^{(k+1)}-Y^*_{:,j}) B_{:,j}^T )\\
& = \frac{\|Y_{:,j}^{(k+1)}-Y^*_{:,j}\|_2^2}{\|B_{:,j}\|_2^4} {\rm trace}( B_{:,j}B_{:,j}^T )= \frac{\|Y_{:,j}^{(k+1)}-Y^*_{:,j}\|_2^2}{\|B_{:,j}\|_2^2}.
\end{align*}
Therefore,
\begin{align*}
\left\|X^{(k+1)}-X^*\right\|_F^2 & = \left\|X^{(k+1)}-\tilde{X}^{(k+1)}\right\|_F^2 + \left\|\tilde{X}^{(k+1)}-X^*\right\|_F^2 =\left\|\tilde{X}^{(k+1)}-X^*\right\|_F^2 + \frac{\|Y_{:,j}^{(k+1)}-Y^*_{:,j}\|_2^2}{\|B_{:,j}\|_2^2}.
\end{align*}
By taking the conditional expectation on both side of this equality, we can obtain
\begin{align}\label{exp}
E_k \left[\left\|X^{(k+1)}-X^*\right\|_F^2 \right] &  = E_k \left[\left\|\tilde{X}^{(k+1)}-X^*\right\|_F^2\right] + E_k \left[\frac{\|Y_{:,j}^{(k+1)}-Y^*_{:,j}\|_2^2}{\|B_{:,j}\|_2^2}\right].
\end{align}

Next, we give the estimates for the first and second parts of the right-hand side of the equality (\ref{exp}) respectively.
If $(X^{(0)}_{i,:})^T\in R(B),\ i=1,\ldots,p$, then $(X^{(0)}-A^+CB^+)^T_{i,:}\in R(B),\ i=1,\ldots,p$. It is easy to show that   $(X^{(k)}-A^+CB^+)^T_{i,:}\in R(B),\ i=1,\ldots,p$ by induction of (\ref{e24}). Then by Lemma \ref{lem4}, we have
\begin{align}\label{exp1}
E_k \left[\|\tilde{X}^{(k+1)}-A^+CB^+\|_F^2\right] \le \rho_2 \left\|X^{(k)}-A^+CB^+\right\|_F^2.
\end{align}
For the second part of the right-hand side of (\ref{exp}),  we have
\begin{align}\label{exp2}
 E_k \left[\frac{\|Y_{:,j}^{(k+1)}-Y^*_{:,j}\|_2^2}{\|B_{:,j}\|_2^2}\right]& = E_k^i E_k^j \left[ \frac{\|Y_{:,j}^{(k+1)}-Y^*_{:,j}\|_2^2}{\|B_{:,j}\|_2^2}\right]\notag\\
 &= E_k^i \left[\sum\limits_{j=1}^n\frac{1}{ \|B \|^2_F}\|Y_{:,j}^{(k+1)}-Y^*_{:,j}\|_2^2 \right]\notag\\
 & = \frac{1}{\|B \|^2_F }E_k^i \left[\left\|Y^{(k+1)} - Y^*\right\|_F^2 \right]\notag\\
 & = \frac{1}{\|B \|^2_F }E_k  \left[\left\|Y^{(k+1)} - Y^*\right\|_F^2 \right].
\end{align}
Substituting (\ref{exp1}), (\ref{exp2}) into (\ref{exp}), we can get
\begin{align*}
E_k \left[\left\|X^{(k+1)}-A^+CB^+\right\|_F^2 \right]
& \le  \rho_2 \left\|X^{(k)}-A^+CB^+\right\|_F^2 + \frac{1}{\|B \|^2_F }E_k  \left[\left\|Y^{(k+1)} - A^+C\right\|_F^2\right],
\end{align*}
Then applying this recursive relation iteratively and taking full expectation, we have
 \begin{align}\label{eq12}
E \left[\left\|X^{(k+1)}-A^+CB^+\right\|_F^2 \right]
& \le \rho_2 E  \left[ \left\|X^{(k)}-A^+CB^+\right\|_F^2\right] + \frac{1}{\|B \|^2_F }E  \left[\left\|Y^{(k+1)} - A^+C\right\|_F^2\right]\notag\\
& \le \rho_2 \left (\rho_2 E  \left[ \left\|X^{(k-1)}-A^+CB^+\right\|_F^2\right]+ \frac{1}{\|B \|^2_F }\rho_1^{k}  \left\|Y^{(0)}-A^+C\right\|_F^2 \right) \notag\\
&\ \ \ \ + \frac{1}{\|B \|^2_F }\rho_1^{k+1}   \left\|Y^{(0)}-A^+C\right\|_F^2 \ ({\rm by\  Lemma}\  \ref{th1})\notag\\& \le \cdots\notag\\
& \le \rho_2^{k+1} \left\|X^{(0)}-A^+CB^+\right\|_F^2 +  \frac{\sum\limits_{j=0}^{k}\rho_1^{j+1}\rho_2^{k-j} }{\|B \|^2_F} \left\|Y^{(0)}-A^+C\right\|_F^2.
\end{align}
Since  $X^*=A^+CB^+$ is the minimal $F$-norm solution of the consistent matrix equation $AXB=C$, so $A A^+CB^+ B = C $. It yields $A^+ A A^+CB^+ B =A^+ C $, then $A^+CB^+ B =A^+ C$. Combining with $Y^{(0)}=X^{(0)}B$, we can get $$\left\|Y^{(0)}-A^+C\right\|_F^2 = \|(X^{(0)}-A^+CB^+)B\|_F^2 \le \sigma^2_{\max}(B)\left\|X^{(0)}-A^+CB^+\right\|_F^2.$$
Substituting this inequality into (\ref{eq12}), then
 \begin{align}\label{eq11}
E \left[\left\|X^{(k+1)}-A^+CB^+\right\|_F^2 \right]
 & \le \left(\rho_2^{k+1}+\frac{ \sigma^2_{\max}(B)}{\|B \|^2_F}  \sum\limits_{j=0}^{k}\rho_1^{j+1}\rho_2^{k-j} \right) \left\|X^{(0)}-A^+CB^+\right\|_F^2 \notag \\
 & = \left(1+\frac{ \sigma^2_{\max}(B)}{\|B \|^2_F} \sum\limits_{j=0}^{k}(\frac{\rho_1}{\rho_2})^{j+1} \right) \rho_2^{k+1} \left\|X^{(0)}-A^+CB^+\right\|_F^2.
\end{align}
If $\rho_1<\rho_2$, then $\sum\limits_{j=0}^{k}(\frac{\rho_1}{\rho_2})^{j+1}  \le  \frac{\frac{\rho_1}{\rho_2}}{1-\frac{\rho_1}{\rho_2}}=\frac{\rho_1}{\rho_2-\rho_1} $. Therefore, (\ref{eq11}) becomes
\begin{align*}
E \left[\left\|X^{(k+1)}-A^+CB^+\right\|_F^2 \right]
& \le \left(1+\frac{ \rho_1 }{\rho_2-\rho_1} \frac{ \sigma^2_{\max}(B) }{\|B \|^2_F} \right) \rho_2^{k+1}\left\|X^{(0)}-A^+CB^+\right\|_F^2.
\end{align*}
If $\rho_1>\rho_2$, $\sum\limits_{j=0}^{k}(\frac{\rho_1}{\rho_2})^{j+1} = \frac{\frac{\rho_1}{\rho_2}-(\frac{\rho_1}{\rho_2})^{k+2}}{1-\frac{\rho_1}{\rho_2}}=\frac{\rho_1}{\rho_1-\rho_2}(\frac{\rho_1}{\rho_2})^{k+1}$.
Therefore, (\ref{eq11}) becomes
\begin{align*}
E \left[\left\|X^{(k+1)}-A^+CB^+\right\|_F^2 \right]
& \le \left(1+\frac{  \rho_1 }{\rho_1-\rho_2}\frac{ \sigma^2_{\max}(B) }{\|B \|^2_F} \left(\frac{\rho_1}{\rho_2}\right)^{k+1} \right) \rho_2^{k+1} \left\|X^{(0)}-A^+CB^+\right\|_F^2.
\end{align*}
If $\rho_1=\rho_2 $, $\sum\limits_{j=0}^{k}(\frac{\rho_1}{\rho_2})^{j+1}= k+1$. Therefore, (\ref{eq11}) becomes
\begin{align*}
E \left[\left\|X^{(k+1)}-A^+CB^+\right\|_F^2 \right]
& \le \left(1+(k+1)  \frac{ \sigma^2_{\max}(B) }{\|B \|^2_F}\right)\rho_2^{k+1} \left\|X^{(0)}-A^+CB^+\right\|_F^2.
\end{align*}
This completes the proof.
\end{proof}

\begin{remark}
Algorithm 2.1 has the advantage that $X^{(k)}$ and $Y^{(k)}$ can be iteratively solved at the same time, or the approximate value of $Y^*=A^+C$ can be iteratively obtained first, and then the approximate value of $X^*=A^+CB^+$ can be iteratively solved.

Generally, if we take $X^{(0)}=0\in R^{p\times q}$ and $Y^{(0)}=0\in R^{p\times n}$, the initial conditions are all satisfied ($0\in R(A^T)$, $0\in R(B)$ and $Y^{(0)}=X^{(0)}B$).
\end{remark}

\section{Coordinate Descent (CD) Method for Inconsistent Case}
If the matrix Eq. (\ref{e11}) is inconsistent, there is no solution to the equation. Now let's consider the least-squares solution of the matrix Eq. (\ref{e11}). Obviously, $X^*=A^+CB^+$ is the unique minimal $F$-norm least-squares solution of the matrix Eq. (\ref{e11}), that is, $$X^*=A^+CB^+=\arg\min\left\{\|X\|_F: \ X\in\arg\min\limits_{X\in R^{p\times q}}\|AXB-C\|_F\right\}.$$

If $A$ is full column rank and $B$ is full row rank, the matrix Eq. (\ref{e11}) has a unique least-squares solution $X^*=(A^TA)^{-1}A^TCB^T(BB^T)^{-1}$. In general, the matrix Eq. (\ref{e11}) has multiple least-squares solutions.
Assume that $A$ has no column of all zeros and $B$ has no row of all zeros. Now we will find $X^*$ with the coordinate descent method (or Gauss-Seidel method).

If a linear system of equations $Ax=b$ is inconsistent, where $A\in R^{m\times p}$ and $r(A)=p$ ($p\leq m$), the RGS (RCD) method \cite{LL10} below is a very effective method to solve its least-squares solution.
\begin{equation}\label{e31}
\alpha_k=\frac{A_{:,j}^T r^{(k)}}{\|A_{:,j}\|_2^2}, \ x^{(k+1)}_j=x^{(k)}_j+\alpha_k, \  r^{(k+1)}=r^{(k)}-\alpha_k A_{:,j}, \ \ \hat{p}_j(A)=\frac{\|A_{:,j}\|_2^2}{\|A\|_F^2},
\end{equation}
where $x^{(0)}\in R^p$ is arbitrary and $r^{(0)}=b-Ax^{(0)}\in R^m$. Simultaneous $n$ iterative formulae for solving $AY_{:,l}=C_{:,l},\ l=1, \cdots, n$, we get
\begin{equation}\label{e32}
W^{(k)}=\frac{A_{:,j}^TR^{(k)}}{\|A_{:,j}\|_2^2}, \ Y^{(k+1)}_{j,:}=Y^{(k)}_{j,:}+W^{(k)}, \  R^{(k+1)}=R^{(k)}-A_{:,j}W^{(k)}, \ \ \hat{p}_j(A)=\frac{\|A_{:,j}\|_2^2}{\|A\|_F^2},
\end{equation}
where $Y^{(0)}\in R^{p\times n}$, $R^{(0)}=C-AY^{(0)}$. This is a column projection method to solve the least-squares solution of $AY=C$ and the cost of each iteration of the method is $4mn+n$ if the square of the column norm of $A$ is pre-computed in advance.

Similarly, we can solve the least-squares solution of $B^TX^T=(Y^{(k+1)})^T$ by the RGS method.
\begin{equation}\label{e33}
U^{(k)}=\frac{E^{(k)}B_{i,:}^T}{\|B_{i,:}\|_2^2},  \ X^{(k+1)}_{:,i}=X^{(k)}_{:,i}+U^{(k)},  \ E^{(k+1)}=E^{(k)}-U^{(k)}B_{i,:}+I_{:,i}W^{(k)},  \ p_i(B)=\frac{\|B_{i,:}\|_2^2}{\|B\|_F^2},
\end{equation}
where $X^{(0)}\in R^{p\times q}$, $E^{(0)}=Y^{(1)}-X^{(0)}B$. This is a row projection method and the cost of each iteration of the method is $4np+n+p$ if the square of the row norm of $B$ is pre-computed in advance.

With (\ref{e32}) and (\ref{e33}), we can get a RGS method for solving (\ref{e11}) as follows, which is called the IME-RGS algorithm.
\begin{algorithm}[H]
  \leftline{\caption{RGS Method for Inconsistent Matrix Equation $AXB=C$ (IME-RGS)\label{alg31}}}
  \begin{algorithmic}[1]
    \Require
      $A\in R^{m\times p}$, $B\in R^{q\times n}$, $C\in R^{m\times n}$, $X^{(0)}\in R^{p\times q}$, $Y^{(0)}=X^{(0)}B$, $R^{(0)}=C-AY^{(0)}$, $E^{(-1)}=0\in R^{p\times n}$, $K\in R$
    \State For $j=1:p$, $M(j)=\|A_{:,j}\|_2^2$
    \State For $i=1:q$, $N(i)=\|B_{i,:}\|_2^2$
    \For {$k=0,1,2,\cdots, K-1$}
    \State Pick $j$ with probability $\hat{p}_{j}(A)=\frac{\|A_{:,j}\|_2^2}{\|A\|^2_F}$   and $i$ with probability $p_{i}(B)=\frac{\|B_{i,:}\|_2^2}{\|B\|^2_F}$
    \State Compute {\small $W^{(k)}=\frac{A_{:,j}^TR^{(k)}}{M(j)}$, $Y^{(k+1)}_{j,:}=Y^{(k)}_{j,:}+W^{(k)}$, $R^{(k+1)}=R^{(k)}-A_{:,j}W^{(k)}$, $E^{(k)}_{j,:}=E^{(k-1)}_{j,:}+W^{(k)}$}
    \State Compute $U^{(k)}=\frac{E^{(k)}B_{i,:}^T}{N(i)}$, $X^{(k+1)}_{:,i}=X^{(k)}_{:,i}+U^{(k)}$, $E^{(k+1)}=E^{(k)}-U^{(k)}B_{i,:}$
    \EndFor
    \State Output $X^{(K)}$
  \end{algorithmic}
\end{algorithm}

In order to prove the convergence of Algorithm \ref{alg31}, we need the following preparations.

\begin{lemma}\label{lem31}
Let $Y^*= A^+C$. The sequence $\{Y^{(k)}\}$ is generated by  (\ref{e32}) starting from the initial matrix $Y^{(0)}\in R^{p\times n}$, then it holds
\begin{equation}\label{e34}
E\left[\|A Y^{(k)}-A A^+C\|_F^2\right]\leq \rho_1^{k}\|AY^{(0)}-AA^+C\|_F^2,
\end{equation}
where the $j$th column of $A$ is selected with probability $\hat{p}_j(A)=\frac{\|A_{:,j}\|_2^2}{\|A\|_F^2}$.
\end{lemma}
\begin{proof}
Since $Y^*= A^+C$ is the least-squares solution of $AY=C$, it yields $A^TAY^*=A^TC$. The coordinate descent updates  (\ref{e32}) for $AY=C$ can be written as
\begin{equation}\label{e35}
Y^{(k+1)}=Y^{(k)} + I_{:,j}\frac{ A_{:,j}^T R^{(k)} }{\|A_{:,j}\|^2_2},
\end{equation}
where $ R^{(k)}=C-AY^{(k)}$.
Then for $k=0,1,2,\cdots,$ we have
\begin{align*}
A_{:,j}^T(AY^{(k+1)}-AY^*) & = A_{:,j}^T(AY^{(k)}+A I_{:,j} \frac{A_{:,j}^T R^{(k)} }{\|A_{:,j}\|^2_2}-AY^*)\\
& = A_{:,j}^T AY^{(k)} +A_{:,j}^TA_{:,j}\frac{A_{:,j}^T R^{(k)} }{\|A_{:,j}\|^2_2}-A_{:,j}^TAY^*\\
& = A_{:,j}^T AY^{(k)} +A_{:,j}^TR^{(k)}-A_{:,j}^TC \ ({\rm by}\ A_{:,j}^TAY^*=A_{:,j}^TC)\\
&=0.
\end{align*}
It  follows that
$$\langle AY^{(k+1)}-AY^{(k)}, AY^{(k+1)}-AY^*\rangle_F=0,$$
and then
\begin{equation}\label{e36}
\|AY^{(k)}-AY^*\|_F^2=\|AY^{(k+1)}-AY^*\|_F^2+\|AY^{(k+1)}-AY^{(k)}\|_F^2.
\end{equation}
Again from (\ref{e35}) we obtain
\begin{align*}
\left\| AY^{(k+1)} - AY^{(k)}\right\|_F^2 & = \left\| A_{:,j}\frac{A_{:,j}^T R^{(k)} }{\|A_{:,j}\|^2_2} \right\|_F^2 \\
& = \frac{1}{\|A_{:,j}\|^4_2}{\rm trace} \left((R^{(k)})^T A_{:,j}A_{:,j}^TA_{:,j}A_{:,j}^TR^{(k)} \right)\\
&= \frac{\|A_{:,j}^TR^{(k)}\|_2^2}{\|A_{:,j}\|^2_2}.
\end{align*}
Substituting this equality into (\ref{e36}) and taking conditional expectation on both sides give
\begin{align}\label{e37}
E_k\left[\|AY^{(k+1)}-AY^*\|_F^2\right]=&E_k\left[\|AY^{(k)}-AY^*\|_F^2-\frac{\|A_{:,j}^TR^{(k)}\|_2^2}{\|A_{:,j}\|^2_2}\right] \notag\\
  & =  \|AY^{(k)}-AY^*\|_F^2 -  \sum\limits_{j=1}^p \frac{\|A_{:,j}\|_2^2}{\|A\|_F^2}\frac{\|A_{:,j}^TR^{(k)}\|_2^2}{\|A_{:,j}\|^2_2}\notag\\
  & = \|AY^{(k)}-AY^*\|_F^2 -\frac{1}{\|A\|_F^2}  \| A^T (AY^*-AY^{(k)})\|_F^2\notag\\
  & \le  \|AY^{(k)}-AY^*\|_F^2 -\frac{ \sigma^2_{\min}(A)}{\|A\|_F^2}  \|  AY^*-AY^{(k)}\|_F^2\notag\\
  & = \left( 1 -\frac{ \sigma^2_{\min}(A)}{\|A\|_F^2} \right) \|  AY^{(k)}- AY^*\|_F^2.
\end{align}
The inequality is obtained by using Lemma \ref{lem3}. Finally, by  (\ref{e37}) and induction on the iteration index $k$, we straightforwardly obtain the estimate (\ref{e34}). This completes the proof.
\end{proof}

\begin{lemma}\label{lem32}
Let $X^*= A^+CB^+$. $\tilde{X}$ is generated by running one-step RGS update for solving the matrix equation $ B^TX^T=( Y^*)^T $ starting from any matrix $\hat{X}\in R^{p\times q}$. Then it holds that
\begin{equation}\label{e378}
E[\|A( \tilde{X}- X^*)B\|_F^2]\leq \rho_2 \|A( \hat{X}- X^*)B\|_F^2,
\end{equation}
where the $i$th row of $B$ is selected with probability $ p_i(B)=\frac{\|B_{i,:}\|_2^2}{\|B\|_F^2}$.
\end{lemma}
\begin{proof}
By the definition of coordinate descent updates for $ B^TX^T=( Y^*)^T $, we have
\begin{equation*}
\tilde{X}=\hat{X} + \frac{ (Y^* - \hat{X} B )B_{i,:}^T }{\|B_{i,:}\|^2_2} I_{i,:}.
\end{equation*}
It yields $A\tilde{X}B=A\hat{X}B+\frac{1}{{\|B_{i,:}\|^2_2}}A(Y^* - \hat{X} B )B_{i,:}^TB_{i,:} $. Using the projection formula satisfied by coordinate descent $B_{i,:}B^T(\tilde{X})^T=B_{i,:}(Y^*)^T$ and the properties of MP generalized inverse, we have
\begin{equation*}
\tilde{X}BB_{i,:}^T =Y^*B_{i,:}^T,\  X^*BB^T = Y^*B^T,\  A^TAX^*BB^T =A^T CB^T.
\end{equation*}
Then
\begin{align*}
\langle A( \tilde{X}- \hat{X})B,  A( \tilde{X} -X^* )B \rangle_F & =\frac{ 1}{\|B_{i,:}\|^2_2}\langle A(Y^* - \hat{X} B )B_{i,:}^TB_{i,:},  A(\tilde{X}-X^*)B \rangle_F \\
& =\frac{ 1}{\|B_{i,:}\|^2_2}  {\rm trace}(B_{i,:}^T B_{i,:}(Y^* - \hat{X} B )^TA^T  A(\tilde{X}-X^*)B  ) \\
& =\frac{ 1}{\|B_{i,:}\|^2_2}  {\rm trace}( A(\tilde{X}-X^*)B B_{i,:}^T B_{i,:}(Y^* - \hat{X} B )^T A^T) \\
 & =0\ ({\rm by}\ \tilde{X}B B_{i,:}^T=X^*B B_{i,:}^T),
\end{align*}
and
\begin{align*}
\| A( \tilde{X} -\hat{X} )B\|_F^2  & =\frac{ 1}{\|B_{i,:}\|^4_2}  \|A(Y^* - \hat{X} B )B_{i,:}^TB_{i,:}\|_F^2 \\
&= \frac{1}{\|B_{i,:}\|_2^4} {\rm trace}( B_{i,:}^T B_{i,:}(Y^* - \hat{X} B )^TA^T  A (Y^* - \hat{X} B )B_{i,:}^T B_{i,:} )\\
&  = \frac{1}{\|B_{i,:}\|_2^4} {\rm trace}( A (Y^* - \hat{X} B )B_{i,:}^T B_{i,:} B_{i,:}^T B_{i,:}(Y^* - \hat{X} B )^TA^T )\\
&  = \frac{1}{\|B_{i,:}\|_2^2} {\rm trace}( A (Y^* - \hat{X} B )B_{i,:}^T  B_{i,:}(Y^* - \hat{X} B )^TA^T )\\
& = \frac{\|A (Y^* - \hat{X} B )B_{i,:}^T\|_2^2}{\|B_{i,:}\|_2^2} =  \frac{\|A (X^* - \hat{X}) B B_{i,:}^T\|_2^2}{\|B_{i,:}\|_2^2}.
\end{align*}
Therefore,
\begin{align}\label{e38}
\|A( \tilde{X} -X^* )B \|_F^2 & = \|A( \hat{X} -X^* )B \|_F^2 - \| A( \tilde{X}- \hat{X})B\|_F^2 \notag\\
 & =\|A( \hat{X} -X^* )B \|_F^2 - \frac{\|A (X^* - \hat{X}) B B_{i,:}^T\|_2^2}{\|B_{i,:}\|_2^2}.
\end{align}
By taking the  expectation on both sides of (\ref{e38}), we can obtain
\begin{align*}
E [\|A( \tilde{X} -X^* )B \|_F^2] & = E \left[  \|A( \hat{X} -X^* )B \|_F^2 -\frac{\|A (X^* - \hat{X}) B B_{i,:}^T\|_2^2}{\|B_{i,:}\|_2^2} \right]\\
 &= \|A( \hat{X} -X^* )B \|_F^2 - \sum\limits_{i=1}^q\frac{\|B_{i,:}\|^2_2}{\|B\|^2_F} \frac{\|A (X^* - \hat{X}) B B_{i,:}^T\|_2^2}{\|B_{i,:}\|_2^2}\\
& =  \|A( \hat{X} -X^* )B \|_F^2 -\frac{\|A (X^* - \hat{X}) B B^T\|_F^2}{\|B\|_F^2}\\
& =  \|A( \hat{X} -X^* )B \|_F^2 -\frac{\|B (A (X^* - \hat{X}) B) ^T\|_F^2}{\|B\|_F^2}\\
 &\le  \|A( \hat{X} -X^* )B \|_F^2 - \frac{\sigma^2_{\min}(B)}{\|B\|_F^2}\| (A (X^* - \hat{X}) B) ^T\|_F^2\\
&= \left( 1 -\frac{\sigma^2_{\min}(B)}{\|B\|_F^2}\right)\| A (X^* - \hat{X}) B \|_F^2 .
\end{align*}
The inequality is obtained by Lemma \ref{lem3} because that  all columns of $(A (X^* - \hat{X}) B) ^T$ are in the range of $B^T$.
This completes the proof.
\end{proof}

\begin{theorem}\label{t31}
Let $\{X^{(k)}\}$ denote the sequence that generated by Algorithm \ref{alg31} for the inconsistent matrix Eq. (\ref{e11}), starting from any initial matrix $X^{(0)}\in R^{p\times q}$ and $Y^{(0)}=X^{(0)}B$. In exact arithmetic, it holds
\begin{equation}\label{e341}
E[\|AX^{(k)}B-AA^+CB^+B\|_F^2]\leq (1+\eta) \rho_2^{k}
\|AX^{(0)}B-AA^+CB^+B\|_F^2+ \eta\rho_2^{k}\|AA^+C-AA^+CB^+B\|_F^2 ,
\end{equation}
where the $j$th column of $A$ is selected with probability $\hat{p}_j(A)=\frac{\|A_{:,j}\|_2^2}{\|A\|_F^2}$
and the $i$th row of $B$ is selected with probability $ p_i(B)=\frac{\|B_{i,:}\|_2^2}{\|B\|_F^2}$ .
\end{theorem}
\begin{proof}
Let $X^{(k)} $ denote the $k$th iterate of RGS method (\ref{e33}) solving $B^TX^T=(Y^{(k+1)})^T$, and $\tilde{X}^{(k+1)}$ be one-step RGS iterate solving $B^TX^T=(Y^*)^T$ from $X^{(k)}$, then
\begin{equation*}
X^{(k+1)}= X^{(k)} + \frac{ (Y^{(k+1)} - X^{(k)} B )B_{i,:}^T }{\|B_{i,:}\|^2_2} I_{i,:},\ \
\tilde{X}^{(k+1)}= X^{(k)} + \frac{ (Y^* - X^{(k)} B )B_{i,:}^T }{\|B_{i,:}\|^2_2} I_{i,:}.
\end{equation*}
Then
\begin{align*}
& \langle A( X^{(k+1)}- \tilde{X}^{(k+1)})B,  A( \tilde{X}^{(k+1)} -X^* )B \rangle_F\\
& =\frac{ 1}{\|B_{i,:}\|^2_2}\langle A(Y^{(k+1)}-Y^* )B_{i,:}^TB_{i,:},  A(\tilde{X}^{(k+1)}-X^*)B \rangle_F \\
& =\frac{ 1}{\|B_{i,:}\|^2_2}  {\rm trace}(B_{i,:}^T B_{i,:}(Y^{(k+1)}-Y^* )^TA^T  A(\tilde{X}^{(k+1)}-X^*)B  ) \\
& =\frac{ 1}{\|B_{i,:}\|^2_2}  {\rm trace}( A(\tilde{X}^{(k+1)}-X^*)B B_{i,:}^T B_{i,:}(Y^{(k+1)}-Y^*)^T A^T) \\
 & =0 \ ({\rm by}\ \tilde{X}^{(k+1)}B B_{i,:}^T = Y^*B_{i,:}^T,\ X^*B B_{i,:}^T= Y^*B_{i,:}^T ),
\end{align*}
and
\begin{align*}
\| A( X^{(k+1)} -\tilde{X}^{(k+1)} )B\|_F^2  & = \left\|\frac{ A(Y^{(k+1)}-Y^*)B_{i,:}^TB_{i,:}}{\|B_{i,:}\|^2_2} \right \|_F^2 \\
&= \frac{1}{\|B_{i,:}\|_2^4} {\rm trace}( B_{i,:}^T B_{i,:}(Y^{(k+1)}-Y^* )^TA^T  A (Y^{(k+1)}-Y^*  )B_{i,:}^T B_{i,:} )\\
&  = \frac{1}{\|B_{i,:}\|_2^4} {\rm trace}( A (Y^{(k+1)}-Y^*  )B_{i,:}^T B_{i,:} B_{i,:}^T B_{i,:}(Y^{(k+1)}-Y^* )^TA^T )\\
& = \frac{\| A (Y^{(k+1)}-Y^*   )B_{i,:}^T\|_2^2}{\|B_{i,:}\|_2^2}  \le \| A (Y^{(k+1)}-Y^*   )\|_F^2.
\end{align*}
Therefore,
\begin{align}\label{e310}
\|A( X^{(k+1)} -X^* )B \|_F^2  &= \|A( X^{(k+1)} -\tilde{X}^{(k+1)} )B \|_F^2 + \| A( \tilde{X}^{(k+1)} -X^*)B\|_F^2 \notag\\
&\le  \| A (Y^{(k+1)}-Y^*   )\|_F^2 +\| A( \tilde{X}^{(k+1)} -X^*)B\|_F^2.
\end{align}
By taking the conditional expectation on both sides of (\ref{e310}), we can obtain
\begin{align*}
E_k \left[\|A( X^{(k+1)} -X^* )B \|_F^2\right]  & \le E_k \left[\|A( \tilde{X}^{k+1} -X^* )B \|_F^2 \right]+ E_k\left[\| A (Y^{(k+1)}-Y^*)\|_F^2 \right] \\
& \le \rho_2 \|A(  X^{(k)} -X^* )B \|_F^2 + \rho_1\| A (Y^{(k)}-Y^* )\|_F^2
\end{align*}
The last inequality is obtained by Lemma \ref{lem31} and Lemma \ref{lem32}.
Applying this recursive relation iteratively, we have
\begin{align}\label{e311}
E \left[\|A( X^{(k+1)} -X^* )B \|_F^2\right]
& \le \rho_2 E \left[\|A( X^{(k)} -X^* )B\|_F^2\right] + \rho_1 E \left[ \|AY^{(k)}-AY^*\|_F^2\right] \notag \\
& \le \rho_2^2 E \left[\|A( X^{(k-1)} -X^* )B\|_F^2 \right] +\rho_1 (\rho_1^{k-1} + \rho_1^k) \|AY^{(0)}-AY^*\|_F^2 \notag \\
& \le \cdots \notag\\
 &\le \rho_2^{k+1} \|A( X^{(0)} -X^* )B\|_F^2 +  \sum\limits_{i=0}^{k}\rho_1^{i+1}\rho_2^{k-i} \|AY^{(0)}-AY^*\|_F^2.
\end{align}
Since $AY^*=AA^+C B^+B + AA^+C(I-B^+B )$ and $ Y^{(0)} = X^{(0)}B$, then
\begin{align*}
\|AY^{(0)}-AY^*\|_F^2  & =  \|AX^{(0)}B-AA^+C B^+B - AA^+C(I-B^+B )\|_F^2\\
& = \|A( X^{(0)} - A^+C B^+) B\|_F^2 +\| AA^+C - AA^+CB^+B \|_F^2  \\
 & \ \ \ \ + 2 \langle A( X^{(0)} - A^+C B^+) B, AA^+C - AA^+CB^+B \rangle_F.
\end{align*}
It follows from
\begin{align*}
\langle A( X^{(0)} - A^+C B^+) B, AA^+C - AA^+CB^+B \rangle_F & = {\rm trace}(B^T( X^{(0)} - A^+C B^+)^TA^T (AA^+C - AA^+CB^+B))\\
&= {\rm trace}(A^T (AA^+C - AA^+CB^+B )B^T( X^{(0)} - A^+C B^+)^T)\\
&= {\rm trace}(A^T AA^+C B^T- A^TAA^+CB^+B B^T)( X^{(0)} - A^+C B^+)^T)\\
&= {\rm trace}(A^T  C B^T- A^T C  B^T)( X^{(0)} - A^+C B^+)^T)  = 0
\end{align*}
that
\begin{align*}
\|AY^{(0)}-AY^*\|_F^2  = \|A( X^{(0)} - A^+C B^+) B\|_F^2 +\| AA^+C - AA^+CB^+B \|_F^2.
\end{align*}
Substituting this equality into (\ref{e311}), we have
\begin{align*}
E \left[\|A( X^{(k+1)} -X^* )B \|_F^2\right]
 & \le\left(\rho_2^{k+1} + \sum\limits_{i=0}^{k}\rho_1^{i+1}\rho_2^{k-i} \right)\|A( X^{(0)} -X^* )B\|_F^2  \\
& \ \ \ \ +  \sum\limits_{i=0}^{k}\rho_1^{i+1}\rho_2^{k-i} \| AA^+C - AA^+CB^+B \|_F^2\\
& \le (1+\eta) \rho_2^{k+1}
\|AX^{(0)}B-AA^+CB^+B\|_F^2+ \eta\rho_2^{k+1}\|AA^+C-AA^+CB^+B\|_F^2 ,
\end{align*}
 where $\eta$ is defined in  Theorem \ref{t21}. This completes the proof.
\end{proof}

\begin{remark}
If $A$ has full column rank and $B$ has full row rank, Theorem \ref{t31} implies that $X^{(k)}$ converges linearly in expectation to $A^+CB^+$. If $A$ does not have full
column rank or $B$ does not have full row rank, Algorithm \ref{alg31} fails to converge (see section 3.3 of the work of Ma et al \cite{N10}.
\end{remark}
\begin{remark}
If the matrix Eq. (\ref{e11}) is consistent, then $ AA^+C = AA^+CB^+B=C $. Therefore, (\ref{e341}) becomes
\begin{equation*}
E [\|AX^{(k)}B -C\|_F^2 ]
\le   (1+ \eta)\rho_2^{k} \|AX^{(0)}B-C\|_F^2.
\end{equation*}
That is, $AX^{(k)}B$ converges to $C$ in expectation (but $X^{(k)}$ does not necessarily converge).
\end{remark}

\begin{remark}
In Algorithm 3.1, $X^{(k)}$ and $Y^{(k)}$ can be iteratively solved at the same time, or the approximate value of $Y^*=A^+C$ can be iteratively obtained first and then the approximate value of $X^*=A^+CB^+$ can be iteratively solved. By using Lemma \ref{lem31} and Lemma \ref{lem32}, we can obtain the similar convergence results. We omit the proof for the sake of conciseness.
\end{remark}

\section{Extended Kacamzrz Method and Extended GS Method for $AXB=C$}
When the matrix Eq. (\ref{e11}) is inconsistent and matrix $A$ or matrix $B$ is not full of rank, using the ideas of \cite{ZF13,Du19,MN15}, we can consider the REK method or REGS method to solve the matrix Eq. (\ref{e11}).

\subsection{$AXB=C$ Inconsistent, $A$ Not Full Rank And $B$ Full Column Rank ($q\geq n$)}
The matrix equation $AY=C$ is solved by the REK method \cite{ZF13, Du19}, while the matrix equation $XB=Y$ is solved by the RK method \cite{SV09}, because $XB=Y$ always has a solution ($B^T$ is full row rank). For this case, we use the  REK-RK method to solve $AXB=C$, which is called the IME-REKRK algorithm.

\begin{algorithm}
  \leftline{\caption{REK-RK Method for Inconsistent Matrix Equation $AXB=C$ (IME-REKRK)\label{alg41}}}
  \begin{algorithmic}[1]
    \Require
      $A\in R^{m\times p}$, $B\in R^{q\times n}$, $C\in R^{m\times n}$, $X^{(0)}=0\in R^{p\times q}$, $Y^{(0)}=0\in R^{p\times n}$, $Z^{(0)}= C$, $K\in R$
    \State For $i=1:m$, $M(i)=\|A_{i,:}\|_2^2$
    \State For $j=1:p$, $N(j)=\|A_{:,j}\|_2^2$
    \State For $l=1:n$, $T(l)=\|B_{:,l}\|_2^2$
    \For {$k=0, 1,\cdots, K-1$}
    \State Pick $i$ with probability $p_{i}(A)=\frac{\|A_{i,:}\|_2^2}{\|A\|^2_F}$ , $j$ with probability $\hat{p}_{j}(A)=\frac{\|A_{:,j}\|_2^2}{\|A\|^2_F}$ and $l$ with probability $\hat{p}_{l}(B)=\frac{\|B_{:,l}\|_2^2}{\|B\|^2_F}$
    \State Compute $Z^{(k+1)}=Z^{(k)}-\frac{A_{:,j}}{N(j)}A_{:,j}^TZ^{(k)}$
    \State Compute $Y^{(k+1)}=Y^{(k)}+\frac{A_{i,:}^T}{M(i)}(C_{i,:}-Z_{i,:}^{(k+1)}-A_{i,:}Y^{(k)})$
    \State Compute $X^{(k+1)}=X^{(k)}+\frac{Y_{:,l}^{(k+1)}-X^{(k)}B_{:,l}}{T(l)}B_{:,l}^T$
    \EndFor
    \State Output $X^{(K)}$
  \end{algorithmic}
\end{algorithm}

\begin{lemma}\label{lem41}
Let $ A \in R^{m\times p}$ and $Y\in R^{p\times n}$, it holds
\begin{align*}
 \sum_{i=1}^m  \frac{\|A_{i,:}\|^2_2}{ \|A\|^2_F}\left\|\left(I-\frac{A^T_{i,:}A_{i,:}}{\|A_{i,:}\|^2_2}\right)Y\right\|_F^2
 =\|Y\|_F^2 -  \frac{\|AY\|_F^2}{\|A\|^2_F}.
\end{align*}
\end{lemma}
\begin{proof}
By the fact of $\|A\|_F^2=\sum_{i=1}^m\|A_{i,:}\|_2^2=\sum_{j=1}^n\|A_{:,j}\|_2^2$ and $\left(I-\frac{A^T_{i,:}A_{i,:}}{\|A_{i,:}\|^2_2}\right)^2=I-\frac{A^T_{i,:}A_{i,:}}{\|A_{i,:}\|^2_2}$, we have
\begin{align*}
 \sum_{i=1}^m  \frac{\|A_{i,:}\|^2_2}{ \|A\|^2_F}\left\|\left(I-\frac{A^T_{i,:}A_{i,:}}{\|A_{i,:}\|^2_2}\right)Y\right\|_F^2
& =\sum_{i=1}^m  \frac{\|A_{i,:}\|^2_2}{ \|A\|^2_F}\sum_{j=1}^n \left\|\left(I-\frac{A^T_{i,:}A_{i,:}}{\|A_{i,:}\|^2_2}\right)Y_{:,j}\right\|_2^2 \\
& =\sum_{i=1}^m\frac{\|A_{i,:}\|^2_2}{ \|A\|^2_F} \sum_{j=1}^n Y_{:,j}^T  \left(I-\frac{A^T_{i,:}A_{i,:}}{\|A_{i,:}\|^2_2}\right)^2 Y_{:,j} \\
& =\sum_{j=1}^n Y_{:,j}^T \left[ \sum_{i=1}^m\frac{\|A_{i,:}\|^2_2}{ \|A\|^2_F} \left(I-\frac{A^T_{i,:}A_{i,:}}{\|A_{i,:}\|^2_2}\right)\right] Y_{:,j} \\
& =\sum_{j=1}^n Y_{:,j}^T  \left(I-\frac{A^T A }{\|A \|^2_F}\right) Y_{:,j} \\
 & = \|Y\|_F^2 -  \frac{\|AY\|_F^2}{\|A\|^2_F}.
\end{align*}
This completes the proof.
\end{proof}

Similar to the proof of Lemma \ref{th1}, we can prove the following Lemma \ref{lem43}.
\begin{lemma}\label{lem43}
Let $Z^*=(I- AA^+)C$. Let $\{Z^{(k)}\}$ denote the $k$th iterate of RK applied to $A^TZ=0$ with the initial guess $Z^{(0)}\in R^{m\times n}$. If $Z^{(0)}_{:,j}\in C_{:,j}+R(A),\ j=1,\ldots, n$,  then $Z^{(k)}$ converges linearly to $(I- AA^+)C$ in mean square form. Moreover, the solution error in expectation for the iteration sequence $Z^{(k)}$ obeys
\begin{equation}\label{e42}
E[\|Z^{(k)}-(I-AA^+)C\|_F^2]\leq \rho_1^{k}\|Z^{(0)}-(I-AA^+)C\|_F^2,
\end{equation}
where the $j$th column of $A$ is selected with probability $\hat{p}_j(A)=\frac{\|A_{:,j}\|_2^2}{\|A\|^2_F}$.
\end{lemma}
\begin{lemma}\label{lem44}
  The sequence $\{Y^{(k)}\}$ is generated by the REK method for $AY=C$  starting from the initial matrix $Y^{(0)}\in R^{p\times n}$ in which $Y^{(0)}_{:,j}\in R(A^T),\ j=1,\ldots, n$ and the initial guess  $Z^{(0)}\in R^{m\times n}$ in which $Z^{(0)}_{:,j}\in C_{:,j}+R(A),\ j=1,\ldots,  n$. In exact arithmetic, it holds
\begin{equation}\label{e43}
E [ \| Y^{(k)} - A^+C\|_F^2 ]   \le \frac{k\rho_1^{k}}{ \|A\|^2_F}\left\|Z ^{(0)}-(I-A A^+)C \right\|_F^2 + \rho_1^{k} \left\|Y^{(0)}-A^+C \right\|_F^2,
\end{equation}
where the $i$th row of $A$ is selected with probability $p_i(A)=\frac{\|A_{i,:}\|_2^2}{ \|A \|^2_F}$,
the $j$th column of $A$ is selected with probability $\hat{p}_j(A)=\frac{\|A_{:,j}\|_2^2}{\|A \|^2_F}$.
\end{lemma}
\begin{proof}
Let ${Y^{(k)}}$  denote the $k$th iterate of REK method for $AY=C$, and $\tilde{Y}^{(k+1)}  $  be   the one-step Kaczmarz update for the matrix equation $AY=AA^+C$ from $Y^{(k)}$, i.e.,
\begin{equation*}
\tilde{Y}^{(k+1)} = Y^{(k)} + \frac{A^T_{i,:}}{\|A_{i,:}\|^2_2} (A_{i,:}A^+C-A_{i,:}Y^{(k)}).
\end{equation*}
 We have
\begin{align*}
\tilde{Y}^{(k+1)} -A^+C  & =Y^{(k)}-A^+C  + \frac{A^T_{i,:}A_{i,:}}{\|A_{i,:}\|^2_2} (A^+C- Y^{(k)})  =  \left(I-\frac{A^T_{i,:}A_{i,:}}{\|A_{i,:}\|^2_2}\right)(Y^{(k)}-A^+C)
\end{align*}
and
\begin{equation*}
 Y^{(k+1)} - \tilde{Y}^{(k+1)} = \frac{A^T_{i,:}}{\|A_{i,:}\|^2_2} (C_{i,:}-Z_{i,:}^{(k+1)}-A_{i,:}A^+C).
\end{equation*}
It follows from
\begin{align*}
& \langle \tilde{Y}^{(k+1)} -A^+C , Y^{(k+1)} - \tilde{Y}^{(k+1)} \rangle_F\\
& = \left\langle  \left(I-\frac{A^T_{i,:}A_{i,:}}{\|A_{i,:}\|^2_2}\right)(Y^{(k)}-A^+C) , \  \frac{A^T_{i,:}}{\|A_{i,:}\|^2_2} (C_{i,:}-Z_{i,:}^{(k+1)}-A_{i,:}A^+C) \right\rangle_F \\
& ={\rm trace}\left((Y^{(k)}-A^+C)^T \left(I-\frac{A^T_{i,:}A_{i,:}}{\|A_{i,:}\|^2_2}\right) \frac{A^T_{i,:}}{\|A_{i,:}\|^2_2} (C_{i,:}-Z_{i,:}^{(k+1)}-A_{i,:}A^+C) \right)\\
&=0 \ ( {\rm by \ } \left(I-\frac{A^T_{i,:}A_{i,:}}{\|A_{i,:}\|^2_2}\right) \frac{A^T_{i,:}}{\|A_{i,:}\|^2_2}=0)
\end{align*}
and
\begin{align*}
\left\|Y^{(k+1)} - \tilde{Y}^{(k+1)}\right\|_F^2 & = \left\|\frac{A^T_{i,:}}{\|A_{i,:}\|^2_2} (C_{i,:}-Z_{i,:}^{(k+1)}-A_{i,:}A^+C)\right\|_F^2 \\
& ={\rm trace}\left((C_{i,:}-Z_{i,:}^{(k+1)}-A_{i,:}A^+C)^T \frac{A _{i,:} }{\|A_{i,:}\|^2_2}  \frac{A^T_{i,:}}{\|A_{i,:}\|^2_2} (C_{i,:}-Z_{i,:}^{(k+1)}-A_{i,:}A^+C) \right)\\
&=\frac{\left\|C_{i,:}-Z_{i,:}^{(k+1)}-A_{i,:}A^+C \right\|_2^2}{\|A_{i,:}\|^2_2}
\end{align*}
that
\begin{align*}
\left\|Y^{(k+1)} - A^+C\right\|_F^2 &  = \left\|Y^{(k+1)} - \tilde{Y}^{(k+1)}\right\|_F^2 +\| \tilde{Y}^{(k+1)} -A^+C \|_F^2 \notag\\
& = \frac{\left\|C_{i,:}-Z_{i,:}^{(k+1)}-A_{i,:}A^+C \right\|_2^2}{\|A_{i,:}\|^2_2} +\| \tilde{Y}^{(k+1)} -A^+C \|_F^2.
\end{align*}
By taking the conditional expectation on the both sides of this equality, we have
\begin{align}\label{e44}
 E_{k} \left[\left\|Y^{(k+1)} - A^+C\right\|_F^2 \right]&  =  E_{k} \left[\frac{\left\|C_{i,:}-Z_{i,:}^{(k+1)}-A_{i,:}A^+C \right\|_2^2}{\|A_{i,:}\|^2_2}\right] + E_{k} \left[\| \tilde{Y}^{(k+1)} -A^+C \|_F^2\right].
\end{align}
Next, we give the estimates for the two parts of the right-hand side of (\ref{e44}). It follows from
\begin{align*}
 E_{k} \left[\frac{\left\|C_{i,:}-Z_{i,:}^{(k+1)}-A_{i,:}A^+C \right\|_2^2}{\|A_{i,:}\|^2_2} \right]
& = E_{k}^j E_{k}^i\left[\frac{\left\|C_{i,:}-Z_{i,:}^{(k+1)}-A_{i,:}A^+C \right\|_2^2}{\|A_{i,:}\|^2_2} \right]\\
& = E_{k}^j\left[ \frac{1}{ \|A\|^2_F}\sum_{i=1}^m \left\|C_{i,:}-Z_{i,:}^{(k+1)}-A_{i,:}A^+C \right\|_2^2 \right]\\
& =  \frac{1}{ \|A\|^2_F}E_{k}^j\left[ \left\|C -Z ^{(k+1)}-A A^+C \right\|_F^2  \right]\\
& =  \frac{1}{ \|A\|^2_F}E_{k}\left[ \left\| Z ^{(k+1)}-(I-A A^+)C \right\|_F^2  \right]
\end{align*}
that
\begin{align}\label{e45}
E \left[\frac{\left\|C_{i,:}-Z_{i,:}^{(k+1)}-A_{i,:}A^+C \right\|_2^2}{\|A_{i,:}\|^2_2} \right]& = \frac{1}{ \|A\|^2_F}E \left[  \left\| Z ^{(k+1)}-(I-A A^+)C \right\|_F^2   \right] \notag\\
& \le \frac{\rho_1^{k+1}}{ \|A\|^2_F}\left\|Z ^{(0)}-(I-A A^+)C \right\|_F^2 \ ({\rm by\ Lemma}\ \ref{lem43}).
\end{align}
By $Y^{(0)}_{:,j}\in R(A^T)$  and $(A^+C)_{:,j}\in R(A^T)$, $j=1,\ldots, n$, we have $(Y^{(0)}-A^+C)_{:,j}\in R(A^T),\ j=1,\ldots, n$. Then, by $Z^{(0)}_{:,j}\in C_{:,j}+R(A)$, it is easy to show that  $Z^{(k)}_{:,j}\in C_{:,j}+R(A)$ and $(Y^{(k)}-A^+C)_{:,j}\in R(A^T),\ j=1,\ldots, n$ by induction. It follows from
\begin{align*}
E_{k}[ \|\tilde{Y}^{(k+1)} -A^+C \|_F^2 ] & = E_{k}^i \left[\left\|\left(I-\frac{A^T_{i,:}A_{i,:}}{\|A_{i,:}\|^2_2}\right)(Y^{(k)}-A^+C)\right\|_F^2 \right]\\
& = \sum_{i=1}^m  \frac{\|A_{i,:}\|^2_2}{ \|A\|^2_F}\left\|\left(I-\frac{A^T_{i,:}A_{i,:}}{\|A_{i,:}\|^2_2}\right)(Y^{(k)}-A^+C)\right\|_F^2\\
& = \left\|Y^{(k)}-A^+C \right\|_F^2-\frac{\|A(Y^{(k)}-A^+C)\|_F^2}{\|A\|^2_F}  \ ({\rm \ by\  Lemma}\ \ref{lem41})\\
&\le \left\|Y^{(k)}-A^+C \right\|_F^2- \frac{\sigma^2_{\min}(A)}{\|A\|^2_F}\left\|Y^{(k)}-A^+C \right\|_F^2 \ ({\rm \ by\  Lemma}\ \ref{lem3})\\
&= \rho_1 \left\|Y^{(k)}-A^+C \right\|_F^2
\end{align*}
that
\begin{align}\label{e46}
E \left[ \left\|\tilde{Y}^{(k+1)} -A^+C \right\|_F^2 \right]\le \rho_1  E \left[ \left\|Y^{(k)}-A^+C \right\|_F^2\right].
\end{align}
Combining (\ref{e44}),(\ref{e45}) and (\ref{e46}) yields
\begin{align*}
E \left[ \| Y^{(k+1)} - A^+C\|_F^2 \right] & = E \left[\frac{\left\|C_{i,:}-Z_{i,:}^{(k+1)}-A_{i,:}A^+C \right\|_2^2}{\|A_{i,:}\|^2_2} \right] + E [ \|\tilde{Y}^{(k+1)} -A^+C \|_F^2 ]\\
&  \le \frac{\rho_1^{k+1}}{ \|A\|^2_F}\left\|Z ^{(0)}-(I-A A^+)C \right\|_F^2 + \rho_1 E \left[ \left\|Y^{(k)}-A^+C \right\|_F^2 \right] \\
& \le \frac{2\rho_1^{k+1}}{ \|A\|^2_F}\left\|Z ^{(0)}-(I-A A^+)C \right\|_F^2 + \rho_1^2 E \left[ \left\|Y^{(k-1)}-A^+C \right\|_F^2 \right] \\
& \le \cdots \le \frac{(k+1)\rho_1^{k+1}}{ \|A\|^2_F}\left\|Z ^{(0)}-(I-A A^+)C \right\|_F^2 + \rho_1^{k+1} \left\|Y^{(0)}-A^+C \right\|_F^2 .
\end{align*}
This completes the proof.
\end{proof}

With these preparations, the convergence proof of algorithm \ref{alg41} is given below.
\begin{theorem}\label{t41}
Let $\{X^{(k)}\}$ denote the sequence that is generated by Algorithm \ref{alg41} ($B$ is full column rank) with the initial guess $X^{(0)}\in R^{p\times q}$ in which $(X^{(0)}_{i,:})^T\in R(B),\ i=1,\ldots, p$. The sequence $\{Y^{(k)}\}$ is generated by the REK method for $AY=C$  starting from the initial matrix $Y^{(0)}=X^{(0)}B$ in which $Y^{(0)}_{:,j}\in R(A^T)$  and $Z^{(0)}\in R^{m\times n}$ in which $Z^{(0)}_{:,j}\in C_{:,j}+R(A),\ j=1,\ldots, n$. In exact arithmetic, it holds
\begin{align}\label{e47}
E \left[ \| X^{(k)} - A^+CB^+\|_F^2 \right]   \le & \left(1+\frac{ \sigma^2_{\max}(B) \eta }{\|B \|^2_F}\right) \rho_2^{k} \left\|X^{(0)}-A^+CB^+ \right\|_F^2 \notag\\
& + \frac{\gamma }{\|B \|^2_F}\rho_2^{k}\left\|Z ^{(0)}-(I-A A^+)C \right\|_F^2 .
\end{align}
where the $i$th row of $A$ is selected with probability $p_i(A)=\frac{\|A_{i,:}\|_2^2}{ \|A \|^2_F}$, the $j$th column of $A$ is selected with probability $\hat{p}_j(A)=\frac{\|A_{:,j}\|_2^2}{\|A\|^2_F}$, the $j$th column of $B$ is selected with probability $ \hat{p}_j(B)=\frac{\|B_{:,j}\|_2^2}{\|B\|_F^2}$ and
\begin{equation*}
  \gamma= \left\{
\begin{array}{ll}
\frac{\rho_1\rho_2 }{(\rho_2-\rho_1)^2}, & if \  \rho_1<\rho_2,\\
\frac{k\rho_1}{\rho_1-\rho_2}\left(\frac{\rho_1}{\rho_2}\right)^{k}, & if \  \rho_1>\rho_2, \vspace{0.2cm}\\
\frac{k(k+1)}{2}  , & if \  \rho_1=\rho_2.
\end{array}
 \right.
\end{equation*}
\end{theorem}
\begin{proof}
Similar to the proof of Theorem \ref{t21} by using Lemma \ref{lem44} and Lemma \ref{lem4}, we can get
\begin{align}\label{e48}
& E \left[\left\|X^{(k+1)}-A^+CB^+\right\|_F^2 \right]
\le \rho_2 E  \left[ \left\|X^{(k)}-A^+CB^+\right\|_F^2\right] + \frac{1}{\|B \|^2_F }E  \left[\|Y^{(k+1)}-A^+C\|_F^2\right]\notag\\
&\ \ \ \   \le \rho_2 E  \left[ \left\|X^{(k)}-A^+CB^+\right\|_F^2\right]   +   \frac{(k+1)\rho_1^{k+1}}{ \|A\|^2_F \|B \|^2_F }\left\|Z ^{(0)}-(I-A A^+)C \right\|_F^2 +\ \frac{\rho_1^{k+1}}{\|B \|^2_F } \left\|Y^{(0)}-A^+C \right\|_F^2\notag\\
 &\ \ \ \    \le \rho_2^2 E  \left[ \|X^{(k-1)}-A^+CB^+\|_F^2\right] +  \frac{\rho_1^{k+1}+\rho_1^k\rho_2}{\|B \|^2_F } \left\|Y^{(0)}-A^+C \right\|_F^2 \notag\\
&\ \ \ \ \ \  \ \ \  + \frac{(k+1)\rho_1^{k+1}+k\rho_1^k\rho_2}{ \|A\|^2_F \|B \|^2_F }\left\|Z ^{(0)}-(I-A A^+)C \right\|_F^2
  \   \le \cdots \notag \\
  &\ \ \ \  \le \rho_2^{k+1} \left\|X^{(0)}-A^+CB^+\right\|_F^2 +  \frac{\sum\limits_{j=0}^{k}\rho_1^{j+1}\rho_2^{k-j} }{\|B \|^2_F} \left\|Y^{(0)}-A^+C\right\|_F^2  \notag \\
 &\ \ \ \ \ \  \ \ \  +  \frac{\sum\limits_{j=0}^{k}(j+1)\rho_1^{j+1}\rho_2^{k-j} }{\|B \|^2_F} \left\|Z ^{(0)}-(I-A A^+)C \right\|_F^2 .
\end{align}
Since  $B$ is  full column rank , then $A^+ C = A^+ CB^+ B $. Combining with $Y^{(0)}=X^{(0)}B$, we can get $$\left\|Y^{(0)}-A^+C\right\|_F^2 = \|(X^{(0)}-A^+CB^+)B\|_F^2 \le \sigma^2_{\max}(B)\left\|X^{(0)}-A^+CB^+\right\|_F^2.$$
Substituting this inequality into (\ref{e48}), then
 \begin{align*}
E \left[\left\|X^{(k+1)}-A^+CB^+\right\|_F^2 \right]
 \le & \left(\rho_2^{k+1}+\frac{ \sigma^2_{\max}(B) }{\|B \|^2_F}  \sum\limits_{j=0}^{k}\rho_1^{j+1}\rho_2^{k-j} \right) \left\|X^{(0)}-A^+CB^+\right\|_F^2 \notag\\
 & +  \frac{ \sum\limits_{j=0}^{k}(j+1)\rho_1^{j+1}\rho_2^{k-j} }{\|B \|^2_F} \left\|Z ^{(0)}-(I-A A^+)C \right\|_F^2\notag\\
 \le & \left(1+\frac{ \eta\sigma^2_{\max}(B) }{\|B \|^2_F}  \right)\rho_2^{k+1} \left\|X^{(0)}-A^+CB^+\right\|_F^2 \notag\\
 &  +  \frac{ \sum\limits_{j=0}^{k}(j+1)(\frac{\rho_1}{\rho_2})^{j+1} }{\|B \|^2_F}\rho_2^{k+1} \left\|Z ^{(0)}-(I-A A^+)C \right\|_F^2.
\end{align*}
If $\rho_1<\rho_2$ , from the  analytical properties of geometric series, we can get $$\sum\limits_{j=0}^{k}(j+1)(\frac{\rho_1}{\rho_2})^{j+1}  \le  \frac{\rho_1\rho_2}{(\rho_2-\rho_1)^2}. $$
If $\rho_1>\rho_2$ , then $$\sum\limits_{j=0}^{k}(j+1)(\frac{\rho_1}{\rho_2})^{j+1} =(\frac{\rho_1}{\rho_2})^{k+1}\sum\limits_{i=0}^{k}(k+1-i)(\frac{\rho_2}{\rho_1})^{i} \le  \frac{(k+1)\rho_1}{\rho_1-\rho_2}(\frac{\rho_1}{\rho_2})^{k+1}.$$
If $\rho_1=\rho_2 $ , then  $$\sum\limits_{j=0}^{k}(j+1)(\frac{\rho_1}{\rho_2})^{j+1} = \frac{(k+1)(k+2)}{2}.$$
This completes the proof.
\end{proof}

\subsection{$AXB=C$ Inconsistent, $A$ Not Full Rank And $B$ Full Row Rank ($q\leq n$)}
The matrix equation $AY=C$ is solved by the REK method, while the matrix equation $XB=Y$ is solved by the RGS method, because $XB=Y$ has a unique least-squares solution ($B^T$ is full column rank). For this case, we call it IME- REKRGS algorithm to solve $AXB=C$, and the algorithm is as follows.
\begin{algorithm}
  \leftline{\caption{REK-RGS Method for Inconsistent Matrix Equation $AXB=C$ (IME-REKRGS)\label{alg44}}}
  \begin{algorithmic}[1]
    \Require
      $A\in R^{m\times p}$, $B\in R^{q\times n}$, $C\in R^{m\times n}$, $X^{(0)}=0\in R^{p\times q}$, $Y^{(0)}=0 \in R^{p\times n}$, $Z^{(0)}=C$, $E^{(0)}=0 \in R^{p\times n}$, $K\in R$
    \State For $i=1:m$, $M(i)=\|A_{i,:}\|_2^2$
    \State For $j=1:p$, $N(j)=\|A_{:,j}\|_2^2$
    \State For $l=1:q$, $T(l)=\|B_{l,:}\|_2^2$
    \For {$k=0, 1,\cdots, K-1$}
    \State Pick $i$ with probability $p_{i}(A)=\frac{\|A_{i,:}\|_2^2}{\|A\|^2_F}$ , $j$ with probability $\hat{p}_{j}(A)=\frac{\|A_{:,j}\|_2^2}{\|A\|^2_F}$ and $l$ with probability $ p_{l}(B)=\frac{\|B_{l,:}\|_2^2}{\|B\|^2_F}$
    \State Compute $Z^{(k+1)}=Z^{(k)}-\frac{A_{:,j}}{N(j)}A_{:,j}^TZ^{(k)}$
    \State Compute $Y^{(k+1)}=Y^{(k)}+\frac{A_{i,:}^T}{M(i)}(C_{i,:}-Z_{i,:}^{(k+1)}-A_{i,:}Y^{(k)})$
    \State Compute $U^{(k)}=\frac{E^{(k)}B_{l,:}^T}{T(l)}$, $X^{(k+1)}_{:,l}=X^{(k)}_{:,l}+U^{(k)}$, $E^{(k+1)}=E^{(k)}-U^{(k)}B_{l,:}$
    \EndFor
    \State Output $X^{(K)}$
  \end{algorithmic}
\end{algorithm}

Similarly, let $Y=AX$, we can transform the equation $AXB=C$ into the system of equations composed of two equations
\begin{equation}
\left\{
\begin{array}{c}
B^TY^T=C^T,\\
AX=Y.
\end{array}
\right.
\end{equation}

The matrix equation $B^TY^T=C^T$ is solved by the RGS method, because it has a unique least-squares solution ($B^T$ is full column rank), while the matrix equation $AX=Y$ is solved by the REK method. For this case, we call it IME-RGSREK algorithm to solve $AXB=C$.

The above two methods can be seen as the combination of two separation algorithms, so we  will not discuss the algorithms in detail,  but only give  the  convergence results of IME-REKRGS method and the proof is omitted.

\begin{theorem}\label{t42}
Let $\{X^{(k)}\}$ denote the sequence that is generated by IME-REKRGS method ($B$ is full row rank) with the initial guess $X^{(0)}\in R^{p\times q}$. The sequence $\{Y^{(k)}\}$ is generated by the REK method for  $AY=C$ starting from the initial matrix $Y^{(0)}=X^{(0)}B $ in which $Y^{(0)}_{:,j}\in R(A^T)$ and the initial gauss $Z^{(0)}\in R^{m\times n}$ in which $Z^{(0)}_{:,j}\in C_{:,j}+R(A),\ j=1,\ldots, n$. In exact arithmetic, it holds
\begin{align}\label{e411}
E [ \|( X^{(k)} - A^+CB^+)B\|_F^2 ]  & \le (1+\eta) \rho_2^{k} \left\|(X^{(0)}-A^+CB^+)B \right\|_F^2 +\eta \rho_2^{k} \left\| A^+C (I-B^+B) \right\|_F^2\notag\\
&\ \ \ \ +\frac{\gamma}{\|A\|_F^2}\rho_2^{k}\left\|Z ^{(0)}-(I-A A^+)C \right\|_F^2 .
\end{align}
where the $i$th row of $A$ is selected with probability $p_i(A)=\frac{\|A_{i,:}\|_2^2}{ \|A \|^2_F}$, the $j$th column of $A$ is selected with probability $\hat{p}_j(A)=\frac{\|A_{:,j}\|_2^2}{\|A\|^2_F}$, the $i$th row of $B$ is selected with probability $ p_i(B)=\frac{\|B_{i,:}\|_2^2}{\|B\|_F^2}$.
\end{theorem}

\subsection{Double Extended Kaczmarz Method for Solving General Matrix Equation $AXB=C$}
In general, the matrix Eq. (\ref{e11}) may be inconsistent, $A$ and $B$ are not full rank, so we consider both matrix equations $AY=C$ and $XB=Y$ are solved by the REK method. The algorithm is described as follows.

\begin{algorithm}
  \leftline{\caption{Double REK Method for general $AXB=C$ (DREK)\label{alg42}}}
  \begin{algorithmic}[1]
    \Require
      $A\in R^{m\times p}$, $B\in R^{q\times n}$, $C\in R^{m\times n}$, $X^{(0)}=0\in R^{p\times q}$, $Y^{(0)}=0\in R^{p\times n}$, $Z^{(0)}=C$, $K_1,\ K_2\in R$
    \State For $i=1:m$, $Mi(i)=\|A_{i,:}\|_2^2$; For $j=1:p$, $Mj(j)=\|A_{:,j}\|_2^2$
    \State For $i=1:q$, $Ni(i)=\|B_{i,:}\|_2^2$; For $j=1:n$, $Nj(j)=\|B_{:,j}\|_2^2$
    \For {$k=0,1,2,\cdots, K_1-1$}
    \State Set $p_i(A)=\frac{\|A_{i,:}\|_2^2}{ \|A \|^2_F}$ and $\hat{p}_j(A)=\frac{\|A_{:,j}\|_2^2}{\|A\|^2_F}$
    \State Compute $Z^{(k+1)}=Z^{(k)}-\frac{A_{:,j}}{N(j)}(A_{:,j}^TZ^{(k)})$
    \State Compute $Y^{(k+1)}=Y^{(k)}+\frac{A_{i,:}^T}{M(i)}(C_{i,:}-Z_{i,:}^{(k+1)}-A_{i,:}Y^{(k)})$
    \EndFor
    \State Set $W^{(0)}=(Y^{(K_1)})^T$
    \For {$k=0,1,2,\cdots, K_2-1$}
    \State Set $p_s(B)=\frac{\|B_{s,:}\|_2^2}{ \|B\|^2_F}$ and $\hat{p}_t(B)=\frac{\|B_{:,t}\|_2^2}{\|B\|^2_F}$
    \State Compute $W^{(k+1)}=W^{(k)}-\frac{B_{s,:}^T}{Ni(s)}B_{s,:}W^{(k)}$
    \State Compute $X^{(k+1)}=X^{(k)}+(Y_{:,t}^{(K_1)}-(W_{t,:}^{(k+1)})^T-X^{(k)}B_{:,t})\frac{B_{:,t}^T}{Nj(t)}$
    \EndFor
    \State Output $X^{(K_2)}$
  \end{algorithmic}
\end{algorithm}

The convergence result is the superposition of the corresponding convergence results of two REK methods and the proof is omitted.

Similar to the DREK method, we can employ the double REGS (DREGS) method to solve general matrix equation $AXB=C$. Since the convergence results and proof methods are very similar to the previous ones, we omit them.
\begin{algorithm}
  \leftline{\caption{Double REGS Method for general $AXB=C$ (DREGS)\label{alg43}}}
  \begin{algorithmic}[1]
    \Require
      $A\in R^{m\times p}$, $B\in R^{q\times n}$, $C\in R^{m\times n}$, $X^{(0)}=0\in R^{p\times q}$, $Y^{(0)}=0\in R^{p\times n}$, $F^{(0)}\in R^{p\times n}$, $U^{(0)}\in R^{p\times q}$, $R^{(0)}=C$, $K_1,\ K_2\in R$
    \State For $i=1:m$, $Mi(i)=\|A_{i,:}\|_2^2$; For $j=1:p$, $Mj(j)=\|A_{:,j}\|_2^2$
    \State For $i=1:q$, $Ni(i)=\|B_{i,:}\|_2^2$; For $j=1:n$, $Nj(j)=\|B_{:,j}\|_2^2$
    \For {$k=0,1,2,\cdots, K_1-1$}
    \State Set $p_i(A)=\frac{\|A_{i,:}\|_2^2}{ \|A \|^2_F}$ and $\hat{p}_j(A)=\frac{\|A_{:,j}\|_2^2}{\|A\|^2_F}$
    \State Compute  $W^{(k)}=\frac{A_{:,j}^T R^{(k)}}{Mj(j)}$, $F^{(k+1)}_{j,:}=F^{(k)}_{j,:}+W^{(k)}$, $R^{(k+1)}=R^{(k)}-A_{:,j}W^{(k)}$
    \State Compute $Y^{(k+1)}=Y^{(k)}-A_{i,:}^T\frac{A_{i,:}(Y^{(k)}-F^{(k+1)})}{Mi(i)}$
    \EndFor
    \State Set $E^{(0)}=Y^{(K_1)}$
    \For {$k=0,1,2,\cdots, K_2-1$}
    \State Set $p_s(B)=\frac{\|B_{s,:}\|_2^2}{ \|B\|^2_F}$ and $\hat{p}_t(B)=\frac{\|B_{:,t}\|_2^2}{\|B\|^2_F}$
    \State Compute $V^{(k)}=\frac{E^{(k)}B_{s,:}^T}{Ni(s)}$, $U^{(k+1)}_{:,i}=U^{(k)}_{:,i}+V^{(k)}$, $E^{(k+1)}=E^{(k)}-V^{(k)}B_{s,:}$
    \State Compute $X^{(k+1)}=X^{(k)}-\frac{(X^{(k)}-U^{(k+1)})B_{:,t}}{Nj(t)}B_{:,t}^T$
    \EndFor
    \State Output $X^{(K_2)}$
  \end{algorithmic}
\end{algorithm}

\section{Numerical Experiments}
In this section, we will present some experiment results of the proposed algorithms for solving various matrix equations, and compare them with ME-RGRK and ME-MWRK in \cite{WLZ22} for consistent matrix equations and RBCD in \cite{DRS22} for inconsistent matrix equations.
 All experiments are carried out by using MATLAB (version R2020a) on a DESKTOP-8CBRR86
with Intel(R) Core(TM) i7-4712MQ CPU @2.30GHz   2.29GHz, RAM 8GB and Windows 10.

All computations are started from the initial guess $X^{(0)}=0, Y^{(0)}=0$, and terminated once the relative error (RE) of the solution, defined by
 $$RE=\frac{\|X^{(k)}-X^*\|_F^2}{\|X^*\|_F^2}$$
 at the the current iterate $X^{(k)}$, satisfies $RE<10^{-6}$ or exceeds maximum iteration $K=50000$, where $X^*=A^+CB^+$.
We report the average number of iterations (denoted as ``IT") and the average computing time in seconds (denoted as``CPU") for 20  trials repeated runs of the corresponding method. We consider the following methods:
\begin{itemize}
\item CME-RK (Algorithm \ref{alg21}), compared with ME-RGRK and ME-MWRK in \cite{WLZ22} for consistent matrix equations. We use $\theta=0.5$ in ME-GRRK method which is the same as in reference \cite{WLZ22}.
\item IME-RGS (Algorithm \ref{alg31}), IME-REKRGS (Algorithm \ref{alg44}), compared with  RBCD in \cite{DRS22} for inconsistent matrix equations. We use $\alpha=\frac{1.5}{\|A\|_2^2}$ in RBCD method, which is the same as in reference \cite{DRS22} .
\item IME-REKRK (Algorithm \ref{alg41}), DREK (Algorithm \ref{alg42}) and DREGS (Algorithm \ref{alg43}) for inconsistent matrix equations where the last two methods have no requirements on whether matrix $A$ and matrix $B$ have full row rank or full column rank.
\end{itemize}
We test the performance of various methods with synthetic dense  data and real-world sparse data. Synthetic data is generated as follows.
\begin{itemize}
\item Type I: For given $m,p,q,n$, the entries of $A$ and $B$ are generated from a standard normal distribution, i.e., $A=randn(m,p), B=randn(q,n).$ We also construct the rank-deficient matrix by $ A=randn(m,p/2), A=[A, A]$ or $ B=randn(q/2,n), B=[B; B]$ and so on.
\item Type II: Like \cite{NZ22}, for given $m,p$, and $r_1=rank(A)$, we construct a matrix $A$ by $A=U_1D_1V_1^T$, where $U_1\in R^{m\times r_1}$ and  $V_1\in R^{p\times r_1}$ are orthogonal columns matrices, $D\in R^{r_1\times r_1}$ is a diagonal matrix whose first $r-2$ diagonal entries are uniformly distributed numbers in $[\sigma_{\min{(A)}}, \sigma_{\max{(A)}}]$, and the last two diagonal entries are $ \sigma_{\max{(A)}},\sigma_{\min{(A)}} $. Similarly, for given $q,n$ and $r_2=rank(B)$, we construct a matrix $B$ by $B=U_2D_2V_2^T$, where $U_2\in R^{q\times r_2}$ and  $V_2\in R^{n\times r_2}$ are orthogonal columns matrices, $D\in R^{r_2\times r_2}$ is a diagonal matrix whose first $r-2$ diagonal entries are uniformly distributed numbers in $[\sigma_{\min{(B)}}, \sigma_{\max{(B)}}]$, and the last two diagonal entries are $ \sigma_{\max{(B)}},\sigma_{\min{(B)}} $.
\end{itemize}
The real-world sparse data come from the Florida sparse matrix collection \cite{DH11}. Table \ref{tab0} lists the  features of these sparse matrices.
\begin{table}[H]
\caption{The detailed features of sparse matrices from \cite{DH11}.}
\label{tab0}
\centering
\begin{tabular}{ c c c c  }
\hline
 name & size  & rank   & sparsity     \\
\hline
  ash219  & $219\times 85$   &  85  & $2.3529\%$  	\\
\hline
 ash958 &  $958\times 292$  &  292  & $0.68493\%$  	\\
\hline
 divorce &  $50\times 9$  &  9  & $50\%$  	\\
\hline
Worldcities &  $315\times 100$  &  100  & $53.625\%$  	\\
\hline
 \end{tabular}
\end{table}

\subsection{Consistent Matrix Equation}
First, we compare the performance of the  ME-RGRK, ME-MWRK  and CME-RK   methods for the consistent matrix equation $AXB=C$. To construct a consistent matrix equation, we set $C=AX^*B$, where $X^*$
 is a random matrix which is generated by $X^*=randn(p,q)$.

 \begin{example}\label{EX5.1} The  ME-RGRK, ME-MWRK  and CME-RK  methods, synthetic dense data. \upshape

 In Table \ref{tab1} and \ref{tab2}, we report the average IT and CPU of  ME-RGRK, ME-MWRK  and CME-RK  for solving consistent matrix with Type I and Type II matrices. In the following tables, the item `>' represents that the number of iteration steps exceeds the maximum iteration (50000), and the item `-' represents that the method does not converge.
From table \ref{tab1}, we can see that the CME-RK  method vastly outperforms the ME-RGRK and ME-MWRK  methods in terms of both IT and CPU times. The CME-RK  method has the least iteration steps and runs the least time regardless of whether the matrices $A$ and $B$ are full column/row  rank or not. We observe that when the linear system is consistent, the speed-up is at least 2.00, and the
biggest reaches 3.75.  As the increasing of matrix dimension, the CPU time of CME-RK method is increasing slowly, while the running time of ME-RGRK and ME-MWRK increases dramatically. The numerical advantages of CME-RK for large consistent matrix equation are more obvious in Table \ref{tab2}. Moreover, when $\frac{\sigma_{\max}(A)}{\sigma_{\min}(A)} $ and $\frac{\sigma_{\max}(B)}{\sigma_{\min}(B)}$ are large (e.g. $\frac{\sigma_{\max}}{\sigma_{\min}}=5$ ),  the convergence speed of  ME-RGRK and ME-MWRK is very slow, because the convergence rate of the two methods depends on $1-\frac{\sigma_{\min}^2(A){\sigma_{\min}^2(B)}}{\|A\|_F^2\|B\|_F^2}$.

Figure \ref{fig1} shows the plots of relative error (RE) in base-10 logarithm versus IT and CPU of different methods  with Type I ($A=randn(500,50), A=[A,A], B=randn(150,600)$ ) and Type II ($m=500, p=100, r_1=50, \frac{\sigma_{\max}(A)}{\sigma_{\min}(A)}=2, q=150, n=600, r_2=50, \frac{\sigma_{\max}(B)}{\sigma_{\min}(B)}=2$). Again, we can see the relative error of CME-RK is decreasing rapidly with the increase of iteration steps and the computing times.

\begin{table}[H]
\caption{IT and CPU of ME-RGRK, ME-MWRK, and CME-RK for the consistent matrix equations with Type I.}
\label{tab1}
\centering
\begin{tabular}{ c c c c  c c c c c c c c}
\hline
  $m$  & $p$  &$r_1$  & $q$  & $n$ &$r_2$ &  & ME-RGRK & ME-MWRK   & CME-RK  \\
\hline
\multirow{2}{*}{100}  & \multirow{2}{*}{40} & \multirow{2}{*}{40}& \multirow{2}{*}{40} & \multirow{2}{*}{100} &\multirow{2}{*}{40} &IT     & 	 49707   & 	27579 	&    1600.9	\\
    &    &    &   &  &   &CPU   &   0.71      &   2.01     &  0.06	\\
\hline
\multirow{2}{*}{100}  & \multirow{2}{*}{40} & \multirow{2}{*}{20}& \multirow{2}{*}{40} & \multirow{2}{*}{100} &\multirow{2}{*}{20} &IT     & 	2979.6    & 1064   	&   454.2  		\\
    &    &    &  &  &    &CPU   &  0.04    &   0.09       &  0.02		\\
\hline
\multirow{2}{*}{40}  & \multirow{2}{*}{100} & \multirow{2}{*}{40}& \multirow{2}{*}{100} & \multirow{2}{*}{40} &\multirow{2}{*}{40} &IT     & 	>   & 49332.7	&   1807.2 		\\
    &    &    &   &  &   &CPU   &   >      &2.15  &   0.13		\\
\hline
\multirow{2}{*}{40}  & \multirow{2}{*}{100} & \multirow{2}{*}{20}& \multirow{2}{*}{100} & \multirow{2}{*}{40} &\multirow{2}{*}{20} &IT     & 	14788    & 3484	 	&   441.1 		\\
    &    &    &   &  &   &CPU   &   0.19      &  0.20      &  0.03		\\
\hline
\multirow{2}{*}{500}  & \multirow{2}{*}{100} & \multirow{2}{*}{100}& \multirow{2}{*}{100} & \multirow{2}{*}{500} &\multirow{2}{*}{100} &IT     & >	 & 	32109  	&  2250.4    \\
    &    &    &   &  &   &CPU   &   >    &  57.61       &  0.33 		\\
\hline
\multirow{2}{*}{500}  & \multirow{2}{*}{100} & \multirow{2}{*}{50}& \multirow{2}{*}{100} & \multirow{2}{*}{500} &\multirow{2}{*}{50} &IT    & 	9193.9 	& 3158.6 &935.3	 \\
    &    &    &   &  &   &CPU   &  6.07 &  5.66    &  0.13 		\\
\hline
\multirow{2}{*}{1000}  & \multirow{2}{*}{200} & \multirow{2}{*}{100}& \multirow{2}{*}{100} & \multirow{2}{*}{1000} &\multirow{2}{*}{50} &IT     & 	15206.4   &  5076.7  	&  1655.5   \\
    &    &    &   &  &   &CPU   &   58.43   &    52.45    &  1.23  	\\
\hline
\multirow{2}{*}{1000}  & \multirow{2}{*}{200} & \multirow{2}{*}{200}& \multirow{2}{*}{100} & \multirow{2}{*}{1000} &\multirow{2}{*}{100} &IT     & 	>   & 39848	  	&  3906.7   \\
    &    &    &   &  &   &CPU   &    >   &   402.15    & 2.51  	\\
\hline
 \end{tabular}
\end{table}
\begin{table}[H]
\caption{IT and CPU of ME-RGRK, ME-MWRK and CME-RK for the consistent matrix equations with Type II.}
\label{tab2}
\centering
\resizebox{\textwidth}{!}{
\begin{tabular}{ c c c c c c c c c c c c c c}
\hline
  $m$  & $p$  &$r_1$ & $\frac{\sigma_{\max}(A)}{\sigma_{\min}(A)}$ & $q$  & $n$ &$r_2$ & $\frac{\sigma_{\max}(B)}{\sigma_{\min}(B)}$  &  & ME-RGRK & ME-MWRK  &  CME-RK \\
\hline
\multirow{2}{*}{100}  & \multirow{2}{*}{40} & \multirow{2}{*}{40} & \multirow{2}{*}{2} & \multirow{2}{*}{40} & \multirow{2}{*}{100} &\multirow{2}{*}{40} & \multirow{2}{*}{2} &IT     & 10865.1	    & 5617	   	&   842.3 		\\
    &    &    &   &  & &&  &CPU   &   0.15      &   0.50        &  0.03		\\
\hline
\multirow{2}{*}{100}  & \multirow{2}{*}{40} & \multirow{2}{*}{20} & \multirow{2}{*}{2} & \multirow{2}{*}{40} & \multirow{2}{*}{100} &\multirow{2}{*}{20} & \multirow{2}{*}{2} &IT     & 	2409    & 	836  	&   422 		\\
    &    &    &   &  & &&  &CPU   &  0.03       &    0.07     &  0.02	 	\\
\hline
\multirow{2}{*}{100}  & \multirow{2}{*}{40} & \multirow{2}{*}{20} & \multirow{2}{*}{5} & \multirow{2}{*}{40} & \multirow{2}{*}{100} &\multirow{2}{*}{20} & \multirow{2}{*}{5} &IT     & 22423.2	    & 6439.8   	&   1145.2		\\
    &    &    &   &  & &&  &CPU   &   0.33     &   0.57      &  0.04		\\
\hline
\multirow{2}{*}{500}  & \multirow{2}{*}{100} & \multirow{2}{*}{100} & \multirow{2}{*}{2} & \multirow{2}{*}{100} & \multirow{2}{*}{500} &\multirow{2}{*}{100} & \multirow{2}{*}{2} &IT     & 40768	    & 20507.2	 	&   1992.3 		\\
    &    &    &   &  & &&  &CPU   &  21.03       &   34.02        &  0.29		\\
\hline
\multirow{2}{*}{500}  & \multirow{2}{*}{100} & \multirow{2}{*}{50} & \multirow{2}{*}{5} & \multirow{2}{*}{500} & \multirow{2}{*}{100} &\multirow{2}{*}{50} & \multirow{2}{*}{5} &IT     & >	    & 	35159  & 2893.6 	\\
    &    &    &   &  & &&  &CPU   &   >      &   59.02       &  0.39		\\
\hline
\multirow{2}{*}{500}  & \multirow{2}{*}{100} & \multirow{2}{*}{50} & \multirow{2}{*}{10} & \multirow{2}{*}{500} & \multirow{2}{*}{100} &\multirow{2}{*}{50} & \multirow{2}{*}{10} &IT     & >	    & 	>  	&  10693.4  	\\
    &   &    &   &  & &&  &CPU   &   >    &     >    &  1.68	\\
\hline
\multirow{2}{*}{1000}  & \multirow{2}{*}{200} & \multirow{2}{*}{100} & \multirow{2}{*}{2} & \multirow{2}{*}{100} & \multirow{2}{*}{1000} &\multirow{2}{*}{50} & \multirow{2}{*}{2} &IT     & 	19679.2    & 6974.9	 	&  1722.8  		\\
    &    &    &   &  & &&  &CPU   &   73.45     &    70.56     &  1.20	\\
\hline
\multirow{2}{*}{1000}  & \multirow{2}{*}{200} & \multirow{2}{*}{100} & \multirow{2}{*}{5} & \multirow{2}{*}{100} & \multirow{2}{*}{1000} &\multirow{2}{*}{50} & \multirow{2}{*}{5} &IT     & 	>  & > 	&  6037.4  		\\
    &    &    &   &  & &&  &CPU   &   >    &    >   & 4.23	\\
\hline
 \end{tabular}}
\end{table}

\begin{figure}\label{fig1}
  \centering
  \subfigure {
    \includegraphics[width=7.5cm]{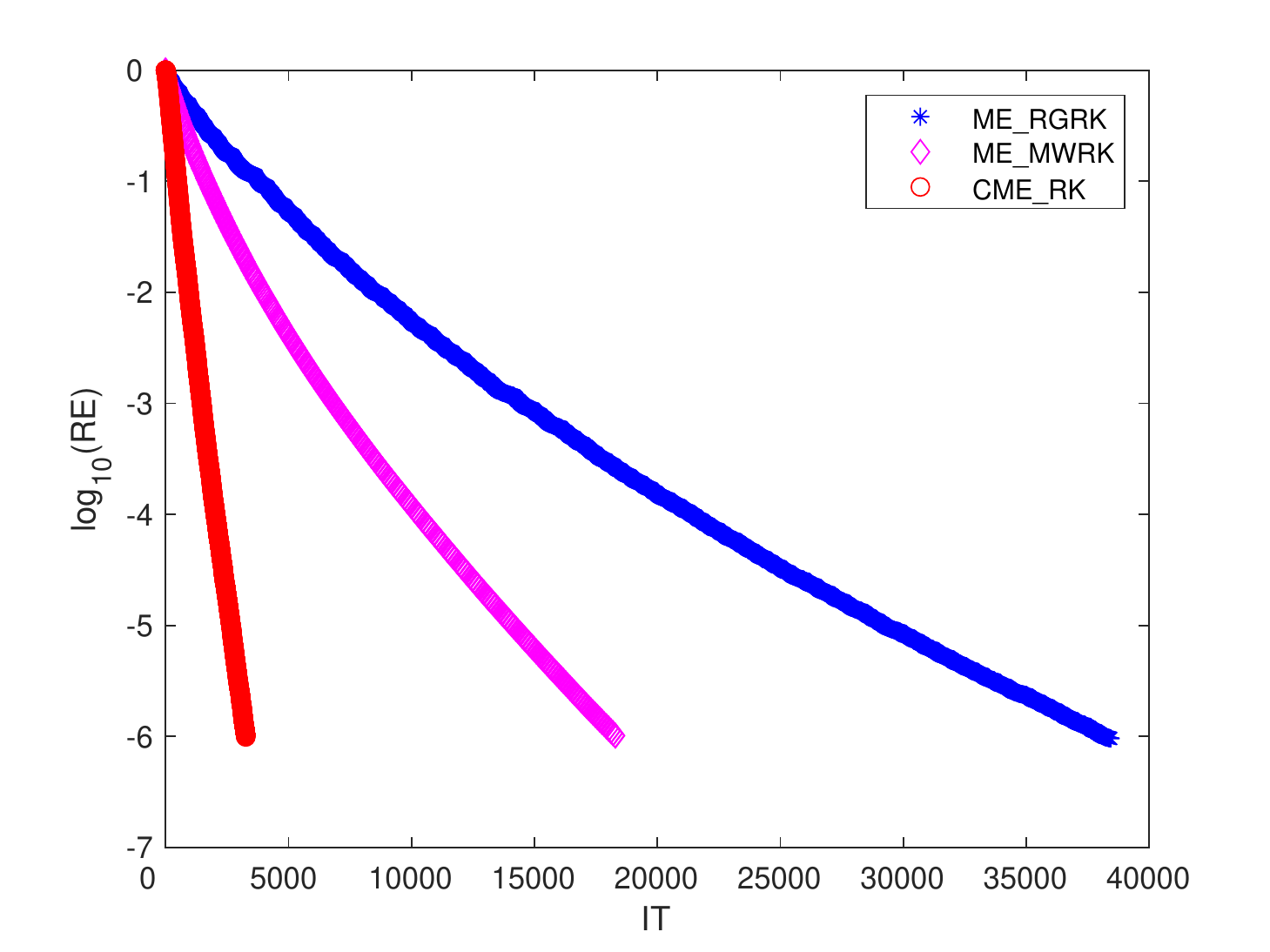}
  }
  \subfigure {
    \includegraphics[width=7.5cm]{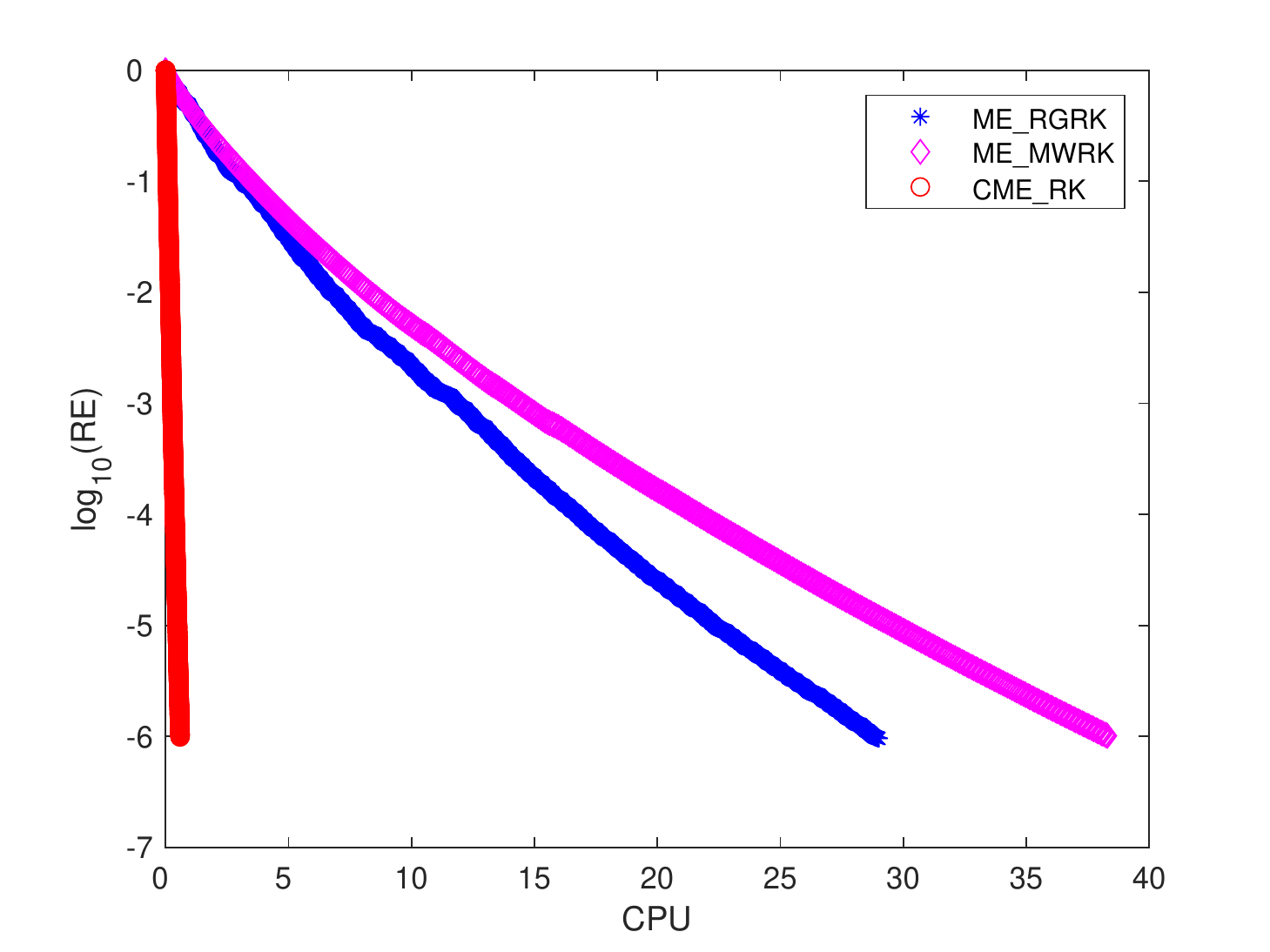}
  }\\
  {\small Type I: $A=randn(500,50), A=[A,A], B=randn(150,600)$}\\
  \subfigure {
    \includegraphics[width=7.5cm]{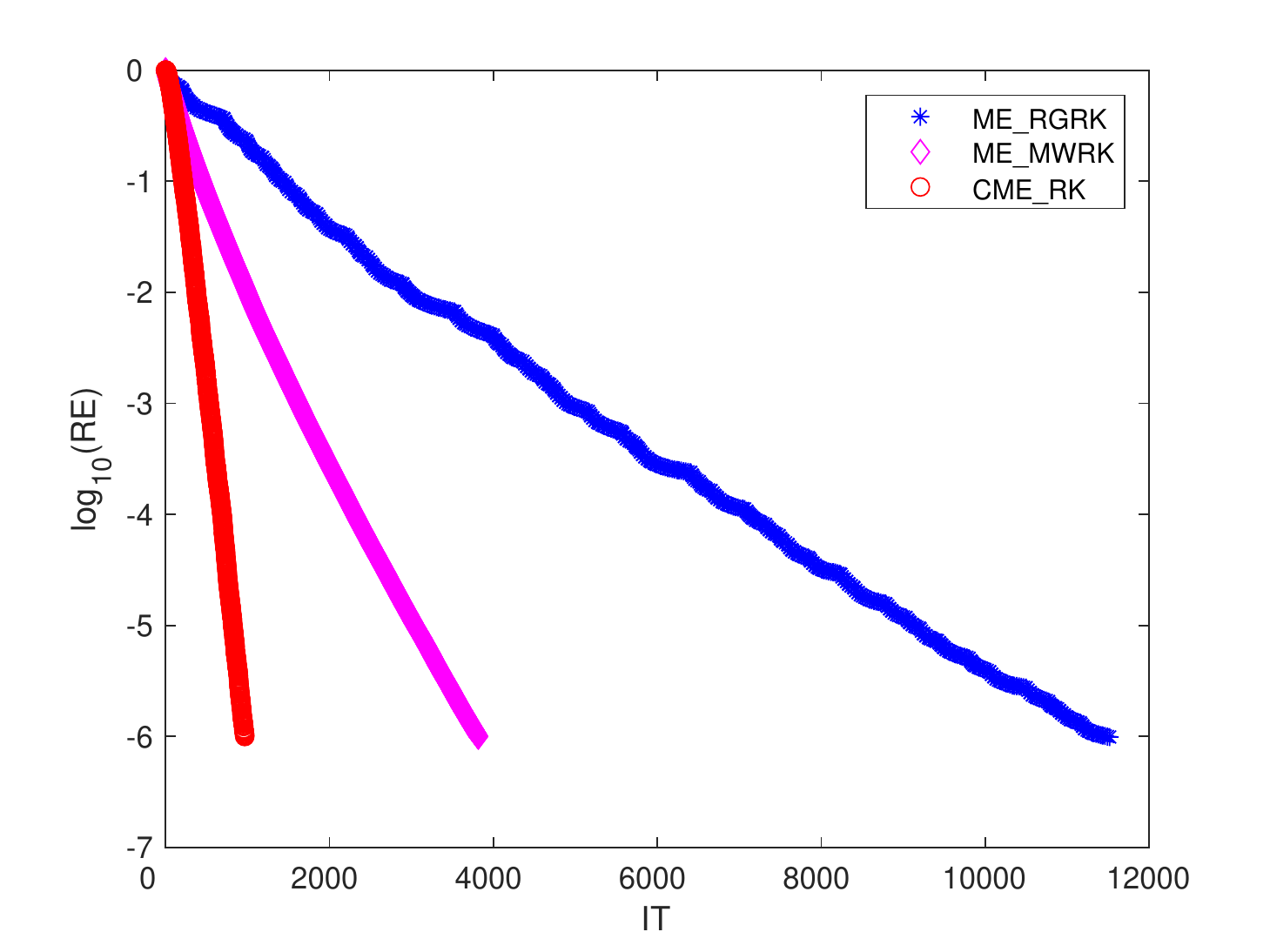}
  }
  \subfigure {
    \includegraphics[width=7.5cm]{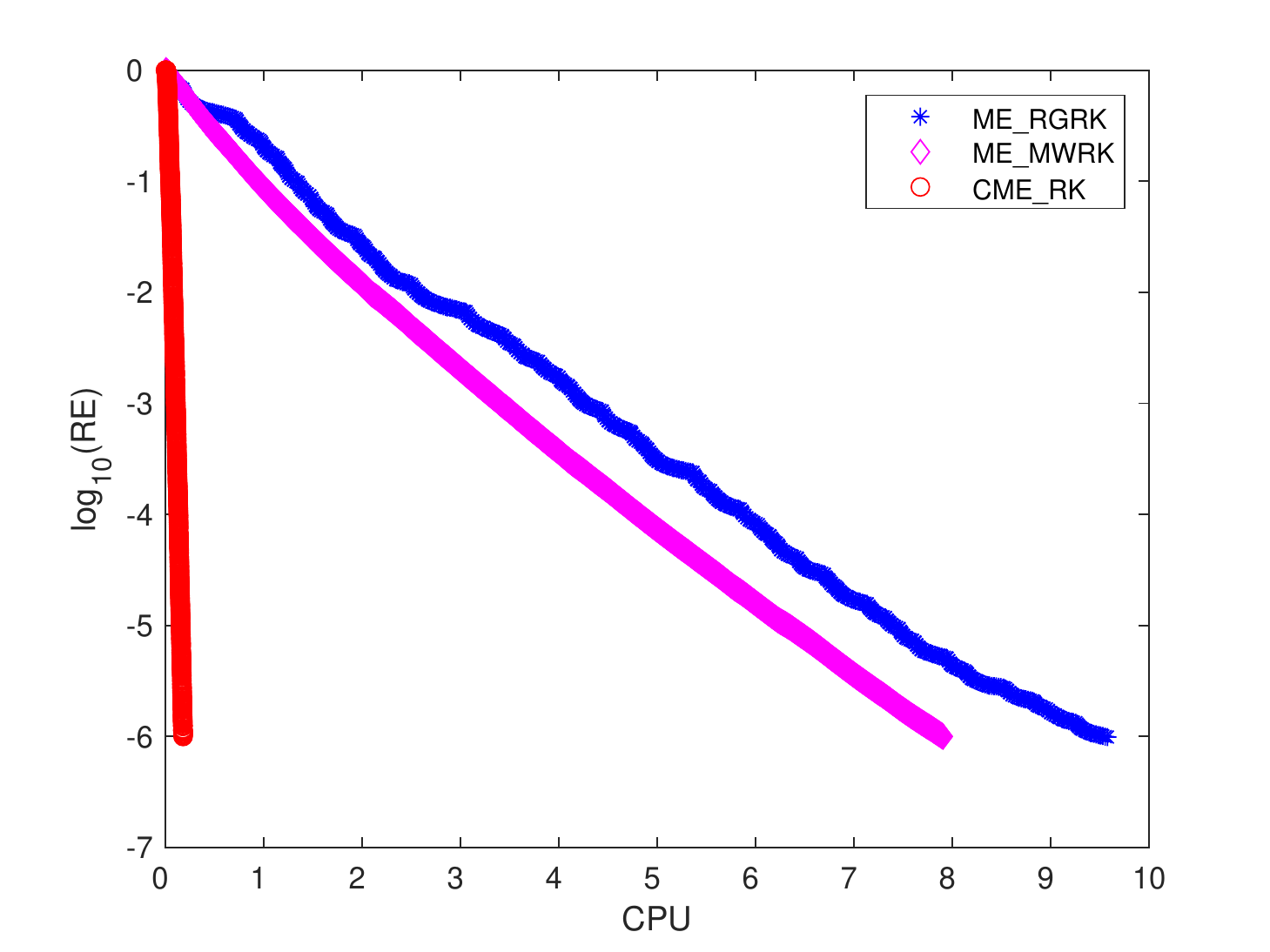}
  }\\
  {\small Type II: $m=500, p=100, r_1=100, \frac{\sigma_{\max}(A)}{\sigma_{\min}(A)}=2, q=150, n=600, r_2=150, \frac{\sigma_{\max}(B)}{\sigma_{\min}(B)}=2$ }\\
  \caption{\emph{\small   IT (left) and CPU (right) of different methods for consistent matrix equations with Type I(top) and Type II (bottom). }}\label{fig1}
\end{figure}

\end{example}

\begin{example}\label{EX5.2} The  ME-RGRK, ME-MWRK  and CME-RK  methods, real-world sparse data. \upshape

For the sparse matrices from \cite{DH11}, we list the numbers of iteration steps and the
computing times for ME-RGRK, ME-MWRK  and  CME-RK  methods in Table \ref{tab3}.  We observe the CME-RK method can  successfully compute an approximate solution of the consistent matrix equation for various $A$ and $B$. For the fist three cases in Table \ref{tab3}, the ME-RGRK, ME-MWRK,  CME-RK methods all converge to the solution, but the CME-RK method is significantly better than the ME-RGRK and ME-MWRK  methods, both in terms of iteration steps and running time. For the last three cases, the ME-RGRK, ME-MWRK methods fail to converge the solution because iteration steps exceed 50000.

\begin{table}[H]
\caption{IT and CPU of ME-RGRK, ME-MWRK,  and CME-RK  for the consistent matrix equations with sparse matrices from \cite{DH11}.}
\label{tab3}
\centering
\begin{tabular}{ c c c c c c}
\hline
  $A$  & $B$   &  & ME-RGRK & ME-MWRK   &  CME-RK   \\
\hline
\multirow{2}{*}{ash219}  & \multirow{2}{*}{divorce$^T$} &IT     & 49871.1	    & 15423.5	   	&   3522.4  	\\
    &     &CPU   &    0.78     &    1.26     &  0.17	 	\\
\hline
\multirow{2}{*}{divorce}  & \multirow{2}{*}{ash219$^T$}  &IT     & 	43927.8    & 14164.4	  	&  3521.3  	 	\\
    &    &CPU   &   1.15      &   1.35          &  0.17 	\\
\hline
\multirow{2}{*}{divorce}  & \multirow{2}{*}{ash219}  &IT     & 	40198.7    & 17251.4	    	&  3238.9 	 	\\
    &    &CPU   &   0.63      &   0.80      &  0.14 	\\
\hline
\multirow{2}{*}{ash958}  & \multirow{2}{*}{ash219$^T$} &IT     & 	>   & 	>   	&  6706.2  	 	\\
    &    &CPU   &  >      &      >       &  1.23	 \\
\hline
\multirow{2}{*}{ash219}  & \multirow{2}{*}{ash958$^T$} &IT     & 	>    & 	>  	&  5762.6	 	\\
    &    &CPU   &   >     &      >        &  1.18	 	\\
\hline
\multirow{2}{*}{ash958}  & \multirow{2}{*}{Worldcities$^T$} &IT     & 	-   & 	> 	&  38088.5  	 	\\
    &    &CPU   &  -      &      >      &  7.94	 \\
\hline
 \end{tabular}
\end{table}
\end{example}
\subsection{Inconsistent Matrix Equation}
Next, we compare the performance of the  RBCD, IME-RGS and IME-REKRGS methods for the inconsistent matrix equation $AXB=C$ where $B$ is full row rank. To construct an inconsistent matrix equation, we set $C=AX^*B+R$, where $X^*$ and $R$ are random matrices which are generated by $X^*=randn(p,q)$ and $R=\delta*randn(p,q), \delta \in (0,1)$. In addition, we also show the experiment results of the REKRK, DREK and DREGS methods, which do not require full row rank of B.

\begin{example}\label{EX5.3} The  RBCD, IME-RGS and IME-REKRGS methods, synthetic dense  data.\upshape

In Table \ref{tab4} and \ref{tab5}, we report the average IT and CPU of the RBCD, IME-RGS and IME-REKRGS methods for solving inconsistent matrix with Type I and Type II matrices. Figure \ref{fig2} shows the plots of relative error (RE) in base-10 logarithm  versus IT and CPU of different methods  with Type I ($A=randn(500,100), B=randn(150,600)$ ) and Type II ($m=500, p=100, r_1=100, \frac{\sigma_{\max}(A)}{\sigma_{\min}(A)}=2, q=150, n=600, r_2=150, \frac{\sigma_{\max}(B)}{\sigma_{\min}(B)}=2$). From these tables, we can see that the IME-RGS and IME-REKRGS methods are better than the RBCD method in terms of IT and CPU time, especially when the matrix dimension is large (see the last two cases in Table \ref{tab4}) or the $\frac{\sigma_{\max}}{\sigma_{\min}}$ is large (see the last three cases in Table \ref{tab5}). From Figure \ref{fig2}, we can find the IME-RGS and IME-REKRGS methods converge faster than the RBCD method, although the relative error of RBCD decreases faster in the initial iteration.
\begin{table}[H]
\caption{IT and CPU of RBCD, IME-RGS and IME-REKRGS for the inconsistent matrix equations with Type I.}
\label{tab4}
\centering
\begin{tabular}{ c c c c  c c c c c c c c}
\hline
  $m$  & $p$  &$r_1$  & $q$  & $n$ &$r_2$ & & RBCD & IME-RGS   & IME-REKRGS  \\
\hline
\multirow{2}{*}{100}  & \multirow{2}{*}{40} & \multirow{2}{*}{40}& \multirow{2}{*}{40} & \multirow{2}{*}{100} &\multirow{2}{*}{40} &IT     & 11116	  & 1883.7	 &    2449.2	  	\\
    &    &    &   &  &   &CPU   &  0.43    &  0.09      &  0.18	 	\\
\hline
\multirow{2}{*}{100}  & \multirow{2}{*}{40} & \multirow{2}{*}{20}& \multirow{2}{*}{40} & \multirow{2}{*}{100} &\multirow{2}{*}{40} &IT     & 	12416    & -	 	&   1725.8 	   \\
    &    &    &   &  &   &CPU   &    0.49     &    -        &  	0.13  	\\
\hline
\multirow{2}{*}{500}  & \multirow{2}{*}{100} & \multirow{2}{*}{100}& \multirow{2}{*}{50} & \multirow{2}{*}{200} &\multirow{2}{*}{50} &IT     & 	2820.1    & 2011.5	  	&   2603.7    \\
    &    &    &   &  &   &CPU   &   0.56      &   0.31        &  0.62 	\\
\hline
\multirow{2}{*}{500}  & \multirow{2}{*}{100} & \multirow{2}{*}{100}& \multirow{2}{*}{100} & \multirow{2}{*}{500} &\multirow{2}{*}{100} &IT     & 	5067    & 2314.7	 	&  2782.2  \\
    &    &    &   &  &   &CPU   &  4.06    &     1.76       &  2.55 	\\
\hline
\multirow{2}{*}{1000}  & \multirow{2}{*}{100} & \multirow{2}{*}{100}& \multirow{2}{*}{200} & \multirow{2}{*}{1000} &\multirow{2}{*}{200} &IT     & 5969.7	    & 3833.4		&   3867.9  \\
    &    &    &   &  &   &CPU   &  27.02    &  18.18   &  19.47	\\
\hline
\multirow{2}{*}{1000}  & \multirow{2}{*}{200} & \multirow{2}{*}{200}& \multirow{2}{*}{200} & \multirow{2}{*}{1000} &\multirow{2}{*}{200} &IT     & 9738.4	    & 4485.2		&   5442   \\
    &    &    &   &  &   &CPU   &   50.43    &  24.27     &  39.64 	\\
\hline
 \end{tabular}
\end{table}
\begin{table}[H]
\caption{IT and CPU of RBCD, IME-RGS and IME-REKRGS  for the inconsistent matrix equations with Type II.}
\label{tab5}
\centering
\resizebox{\textwidth}{!}{
\begin{tabular}{ c c c c c c c c c c c c c}
\hline
  $m$  & $p$  &$r_1$ & $\frac{\sigma_{\max}(A)}{\sigma_{\min}(A)}$ & $q$  & $n$ &$r_2$ & $\frac{\sigma_{\max}(B)}{\sigma_{\min}(B)}$  &  & RBCD & IME-RGS    &   IME-REKRGS \\
\hline
\multirow{2}{*}{100}  & \multirow{2}{*}{40} & \multirow{2}{*}{40} & \multirow{2}{*}{2} & \multirow{2}{*}{40} & \multirow{2}{*}{100} &\multirow{2}{*}{40} & \multirow{2}{*}{2} &IT     & 	1176.5    & 716.2	   	&   949.6	\\
    &    &    &   &  & &&  &CPU   & 0.04	    & 0.03	  	&   0.09		\\
\hline
\multirow{2}{*}{100}  & \multirow{2}{*}{40} & \multirow{2}{*}{40} & \multirow{2}{*}{5} & \multirow{2}{*}{40} & \multirow{2}{*}{100} &\multirow{2}{*}{40} & \multirow{2}{*}{5} &IT      & 27631	    & 	2974.3   	&  3773.5 	\\
    &    &    &   &  & &&  &CPU   & 	1.10    & 0.16	  	&   0.27		\\
\hline
\multirow{2}{*}{500}  & \multirow{2}{*}{100} & \multirow{2}{*}{100} & \multirow{2}{*}{2} & \multirow{2}{*}{100} & \multirow{2}{*}{500} &\multirow{2}{*}{100} & \multirow{2}{*}{2} &IT     & 2953.4	    & 2101.2	  &  2307.8 		\\
    &    &    &   &  & &&  &CPU    & 	2.52    & 1.78	   	&   	2.25	\\
\hline
\multirow{2}{*}{500}  & \multirow{2}{*}{100} & \multirow{2}{*}{100} & \multirow{2}{*}{5} & \multirow{2}{*}{100} & \multirow{2}{*}{500} &\multirow{2}{*}{100} & \multirow{2}{*}{5} &IT      & >	    & 	6432.5  	&  8101.3 	\\
    &    &    &   &  & &&  &CPU    & 	>    & 5.29	  	&   7.96		\\
\hline
\multirow{2}{*}{1000}  & \multirow{2}{*}{100} & \multirow{2}{*}{100} & \multirow{2}{*}{2} & \multirow{2}{*}{200} & \multirow{2}{*}{1000} &\multirow{2}{*}{200} & \multirow{2}{*}{2} &IT     & 5577.3	    & 3242.6	   &   3380.7	\\
    &   &    &   &  & &&  &CPU    & 	25.68    & 15.33	  &   17.13	\\
\hline
 \multirow{2}{*}{1000}  & \multirow{2}{*}{100} & \multirow{2}{*}{100} & \multirow{2}{*}{5} & \multirow{2}{*}{200} & \multirow{2}{*}{1000} &\multirow{2}{*}{200} & \multirow{2}{*}{5} &IT    & >	    & 10672.5	  	&   11006.4		\\
    &    &    &   &  & &&  &CPU    & >	    & 	49.49  	&  56.54	\\
 \hline
 \end{tabular}}
\end{table}

\begin{figure}[H]\label{fig2}
  \centering
  \subfigure {
    \includegraphics[width=7.5cm]{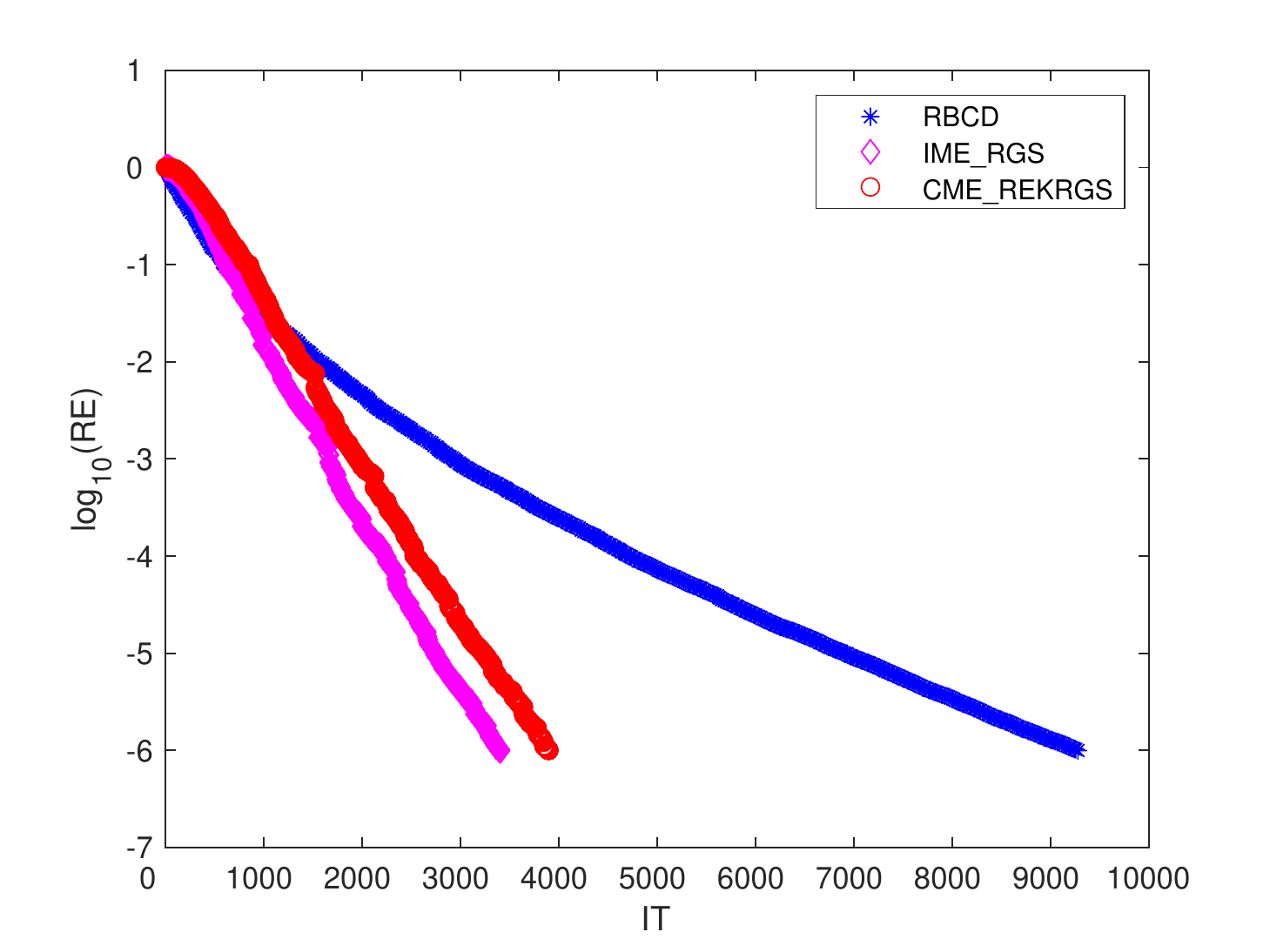}
  }
  \subfigure {
    \includegraphics[width=7.5cm]{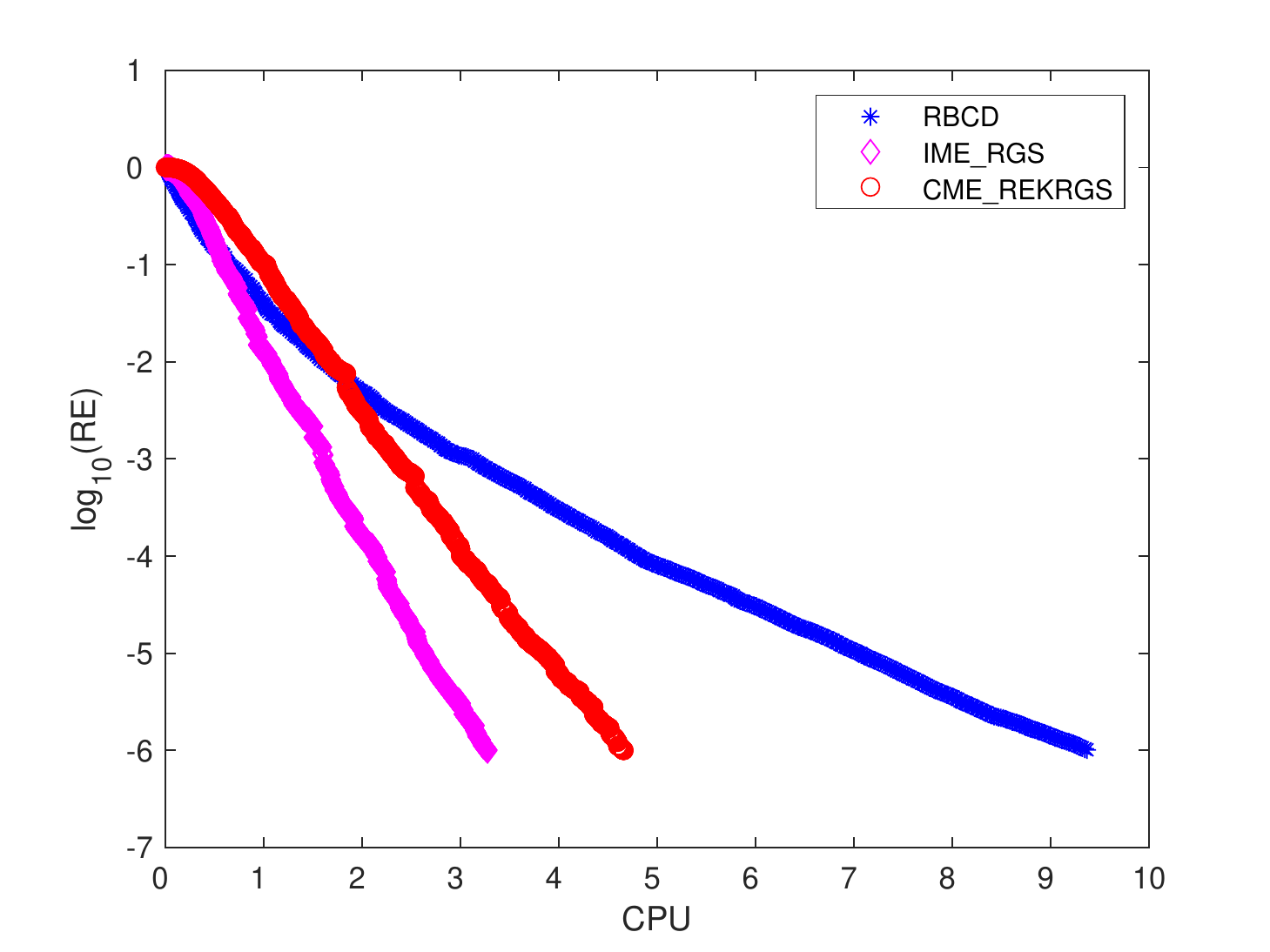}
  }\\
  {\small Type I: $A=randn(500,100), B=randn(150,600)$} \\
  \subfigure {
    \includegraphics[width=7.5cm]{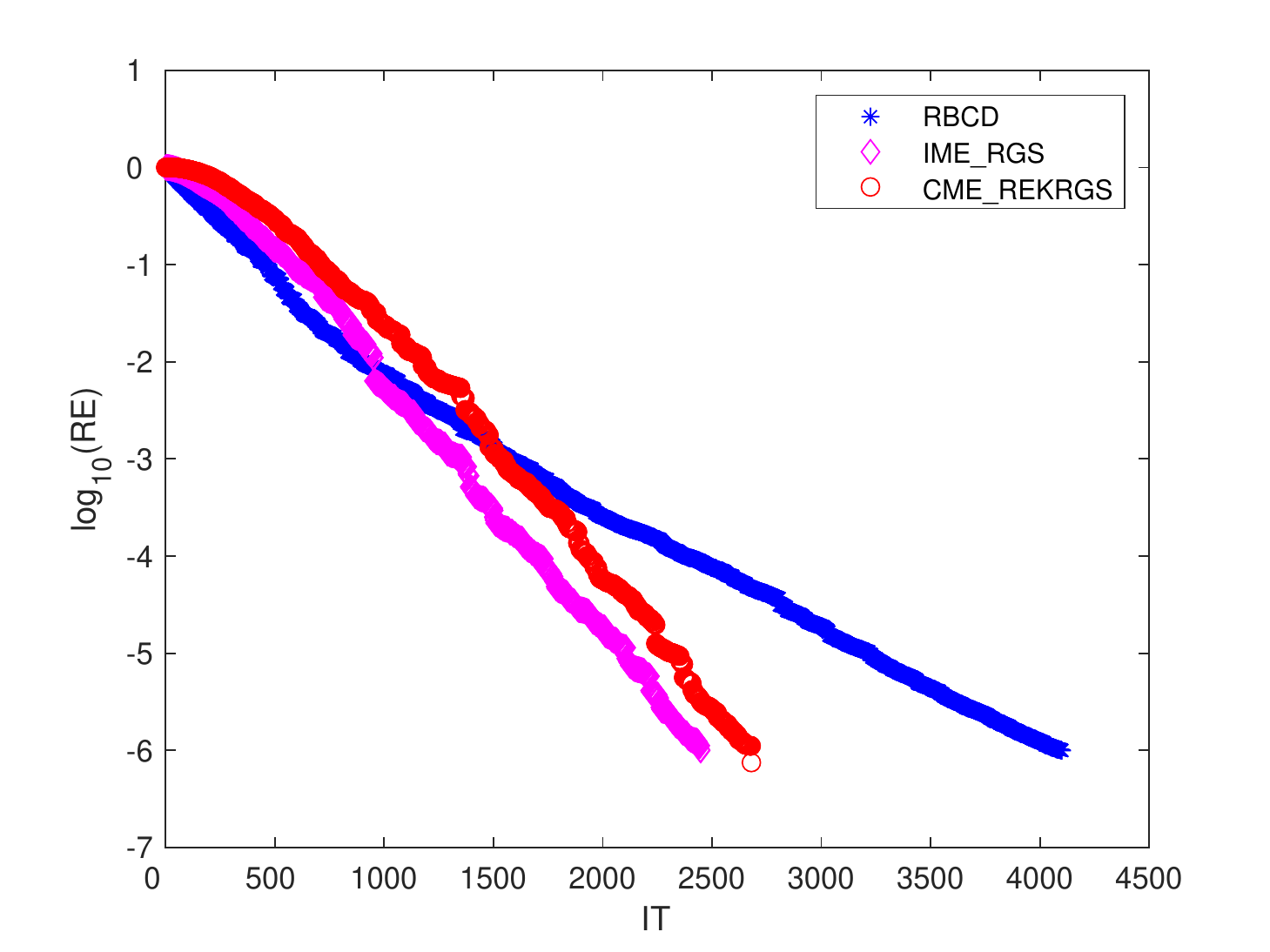}
  }
  \subfigure {
    \includegraphics[width=7.5cm]{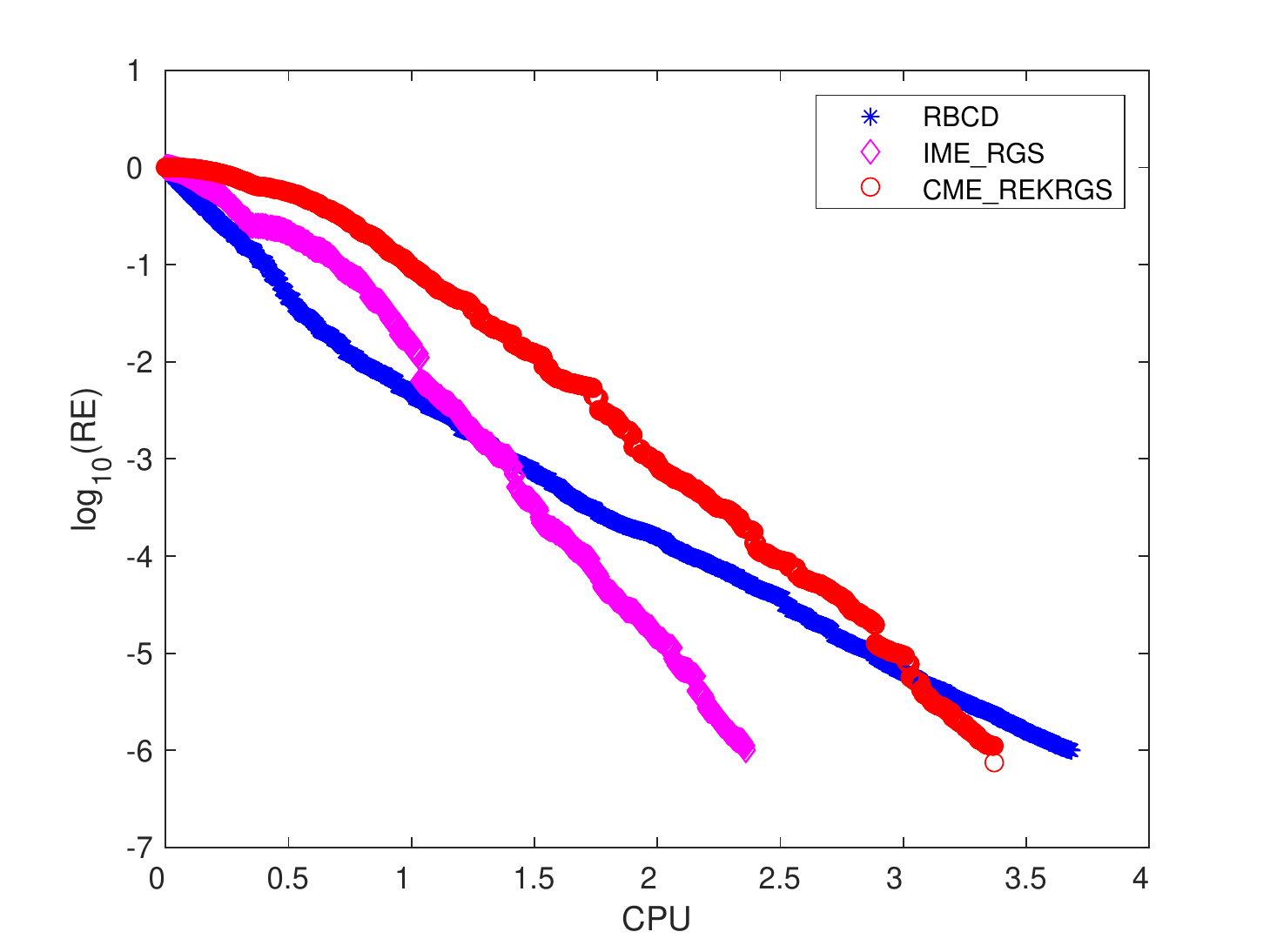}
  }\\
  {\small Type II: $m=500, p=100, r_1=100, \frac{\sigma_{\max}(A)}{\sigma_{\min}(A)}=2, q=150, n=600, r_2=150, \frac{\sigma_{\max}(B)}{\sigma_{\min}(B)}=2$ }\\
  \caption{\emph{\small   IT (left) and CPU (right) of different methods for inconsistent matrix equations with Type I(top) and Type II (bottom). }}\label{fig2}
\end{figure}

\end{example}

\begin{example}\label{EX5.4} The  RBCD, IME-RGS and IME-REKRGS methods, real-world sparse data. \upshape

In table \ref{tab6}, we list the average IT and CPU  of the RBCD, IME-RGS  and IME-REKRGS methods for solving inconsistent matrix with sparse matrices. We can observe that the IME-RGS  and IME-REKRGS methods require less CPU time than the RBCD method in all case and less IT in all case except for $A=ash219, B=ash958^T$.
\begin{table}[H]
\caption{IT and CPU of RBCD, IME-RGS  and IME-REKRGS for the inconsistent matrix equations with sparse matrices from \cite{DH11}.}
\label{tab6}
\centering
\begin{tabular}{ c c c c c c c }
\hline
  $A$  & $B$   &  & RBCD & IME-RGS    &  IME-REKRGS \\
\hline
\multirow{2}{*}{ash219}  & \multirow{2}{*}{divorce$^T$} &IT     & 13115.2	    & 3543.7	 	&   3653.6 	\\
    &     &CPU   &    1.15   &   0.20  &  0.28		\\
\hline
\multirow{2}{*}{divorce}  & \multirow{2}{*}{ash219$^T$}  &IT     & 	 >   & 	3371.8  	&  4150.9  	  	\\
    &    &CPU   &   >      &    0.26       &  	0.48	\\
\hline
\multirow{2}{*}{ash958}  & \multirow{2}{*}{ash219$^T$}  &IT     & 	7200   & 	7536.4  	&  7808.6  	  	\\
    &    &CPU   &   10.44     &  5.06      &  	6.04	\\
\hline

\multirow{2}{*}{ash219}  & \multirow{2}{*}{ash958$^T$}  &IT     & 21118.5	  & 6749.2	   &  5675.7  	  	\\
    &    &CPU   &  14.93      &  5.37     &  6.36	\\
\hline
\multirow{2}{*}{ash958}  & \multirow{2}{*}{Worldcities$^T$}  &IT     & 	>   & 	39183.6  	&  38538.1	  	\\
    &    &CPU   &   >     &  42.34    &  	53.65	\\
\hline
 \end{tabular}
\end{table}
\end{example}

\begin{example}\label{EX5.5} The IME-REKRK, DREK and DREGS methods. \upshape

Finally, we test the effectiveness of IME-REKRK, DREK and DREGS methods for inconsistent matrix equations,  including synthetic dense data and real-world sparse data. The features of $A$ and $B$ are given in Table \ref{tab7}. The experiment results are listed in Table \ref{tab8}. For the DREK and GREGS methods, the iteration steps and running time for calculating $Y^{(k)}$ and $X^{(k)}$ are represented by ``+''. From Table \ref{tab8}, we can observe that the IME-REKRK method can compute an approximate solution to the linear least-squares problems when B is full column rank. The DREK and DREGS methods can successfully solve the linear least-squares solution for all cases.
\begin{table}[H]
\caption{The details feature of $A$ and $B$ of Example \ref{EX5.5}.}
\label{tab7}
\centering
\resizebox{\textwidth}{!}{
\begin{tabular}{ c c c c c }
\hline
 Type  & $A$  & $r(A)$  & $B$ & $r(B)$   \\
\hline
  Type a & $A=randn(500,100), A=[A,A;A,A]$ & 100  & $B=randn(1000,100)$    &  100 		\\
\hline
 Type b & $A=randn(500,100), A=[A,A;A,A]$ & 100  & $B=randn(50,500),B=[B,B;B,B]$    &  50 		\\
\hline
 Type c &  \makecell[c]{$A=U_1D_1V_1^T,$ \\ $m=1000,p=100,\frac{\sigma_{\max}(A)}{\sigma_{\min}(A)}=2$}  & 50  & \makecell[c]{$B=U_2D_2V_2^T,$ \\ $q=200,n=1000,\frac{\sigma_{\max}(B)}{\sigma_{\min}(B)}=2$}   &  40 	\\
\hline
Type d &  \makecell[c]{$A=U_1D_1V_1^T,$ \\ $m=1000,p=100,\frac{\sigma_{\max}(A)}{\sigma_{\min}(A)}=5$}  & 100  & \makecell[c]{$B=U_2D_2V_2^T,$ \\ $q=1000,n=100,\frac{\sigma_{\max}(B)}{\sigma_{\min}(B)}=5$}   &  100 	\\
\hline
Type e &   A=ash219 &  85  & B=Worldcities  &  100 	\\
\hline
Type f & A=ash958 &  292  & B=ash219 &  85  	\\
\hline
 \end{tabular}}
\end{table}
\begin{table}[H]
\caption{IT and CPU of the IME-REKRK, DREK and DREGS methods for the inconsistent matrix equations.}
\label{tab8}
\centering
\resizebox{\textwidth}{!}{
\begin{tabular}{ c c c c c c c c c }
\hline
 Method & & Type a &  Type b  &  Type c &  Type d &  Type e  &  Type f    \\
\hline
\multirow{2}{*}{IME-REKRK} &IT  &  2692.3 & 2569.4   &  -    & -  & 41373.1   &  8733.5    \\
  & CPU   & 3.81  & 14.52   &  -  & - & 8.36   &  2.60   \\
\hline
\multirow{2}{*}{DREK}&IT  & 1785.2 + 1981.7   &  2110.5 + 1067.2  & 1171.6 + 942.1    & 3188.2 + 2962.7   & 1831.3 + 596.6      &  7780.7 + 3236.6 	\\
  & CPU   & 0.60 + 1.98  &  10.87 + 1.51  &   5.88 + 0.97 & 13.66 + 1.21 &  0.33 + 0.16    &  1.56 + 0.62   \\
 \hline
\multirow{2}{*}{DREGS}&IT & 2087.3 + 1792.6  &   2358.2 + 1199.5  & 1327.6 + 1194.8 & 3205.6 + 3528.1 &  2498.2 + 703.9 &  7998.1 + 3568.5   \\
 & CPU   &  0.74 + 3.26 & 12.36 + 1.54   & 6.97 + 1.22 & 15.36 + 1.47 &  0.30 + 0.20  &  1.69 + 0.71    \\
\hline
 \end{tabular}}
\end{table}
\end{example}
\section{Conclusion}
For consistent matrix equations $AXB=C$, we have proposed a Kaczmarz-type algorithm; for inconsistent case, we have suggested a Gauss-Seidel algorithm if $A$ is full column rank and $B$ is full row rank; and for the matrix is inconsistent and $A$ or $B$ is not full rank, we have given out some extended Kaczmarz and extended Gauss-Seidel algorithms. Theoretically, we have proved the proposed algorithms converge linearly to the unique minimal $F$-norm solution or least-squares solution (i.e., $A^+CB^+$) and numerical results show the effectiveness of all algorithms.

\end{document}